\documentclass[12pt,reqno]{amsart}
\usepackage[utf8]{inputenc}
\usepackage[T1]{fontenc}
\usepackage{amsmath}
\usepackage{amsthm}
\usepackage{amssymb}
\usepackage[abbrev]{amsrefs}
\usepackage{mathrsfs}
\usepackage[dvipsnames]{xcolor}
\usepackage{bm}
\usepackage{enumitem}
\usepackage{hyperref}
\usepackage{graphicx}
\usepackage[all]{xy}
\AtBeginDocument{\def\MR#1{}}
\makeatletter
\@namedef{subjclassname@2020}{%
\textup{2020} Mathematics Subject Classification}
\makeatother
\numberwithin{equation}{section}
\allowdisplaybreaks
\newtheorem{thmintro}{}

\newtheorem{theoremintro}[thmintro]{Theorem}
\newtheorem{questionintro}[thmintro]{Question}
\newtheorem{thm}{}[section]
\newtheorem{theorem}[thm]{Theorem}
\newtheorem{corollary}[thm]{Corollary}
\newtheorem{lemma}[thm]{Lemma}
\newtheorem{proposition}[thm]{Proposition}
\newtheorem{question}[thm]{Question}
\theoremstyle{definition}

\newtheorem{remark}[thm]{Remark}
\let\tq=\colon
\let\tp=\colon
\newcommand{\NN}{\ensuremath{\mathbb{N}}}
\newcommand{\RR}{\ensuremath{\mathbb{R}}}
\newcommand{\FF}{\ensuremath{\mathbb{F}}}
\newcommand{\QQ}{\ensuremath{\mathbb{Q}}}
\newcommand{\ZZ}{\ensuremath{\mathbb{Z}}}

\newcommand{\Dt}{\ensuremath{\mathcal{D}}}

\newcommand{\Et}{\ensuremath{\mathcal{E}}}
\newcommand{\Xt}{\ensuremath{\mathcal{X}}}

\newcommand{\Ft}{\ensuremath{\mathcal{F}}}
\newcommand{\St}{\ensuremath{\mathcal{S}}}

\newcommand{\Id}{\ensuremath{{\rm{Id}}}}
\DeclareMathOperator*{\co}{co}
\DeclareMathOperator*{\cocl}{\overline{co}}
\DeclareMathOperator*{\essinf}{ess\,inf}
\DeclareMathOperator*{\esssup}{ess\,sup}
\DeclareMathOperator{\lip}{lip}
\DeclareMathOperator{\Lip}{Lip}
\DeclareMathOperator{\LipD}{DL}
\DeclareMathOperator{\md}{md}
\DeclareMathOperator{\supp}{supp}

\newcommand{\abs}[1]{\left\lvert#1\right\rvert}
\newcommand{\norm}[1]{\left\lVert#1\right\rVert}

\newcommand{\enbrace}[1]{\left\lbrace#1\right\rbrace}

\newcommand{\enpar}[1]{\left(#1\right)}
\newcommand{\littletaller}{\mathchoice{\vphantom{\big|}}{}{}{}}
\newcommand\restr[2]{{\left.\kern-\nulldelimiterspace #1 \littletaller \right|_{#2}}}

\author[F. Albiac]{Fernando Albiac}
\address{Institute for Advanced Materials and Mathematics (INAMAT$^{2}$) and Department of Mathematics, Statistics and Computer Sciences\\ Public University of Navarre\\
Campus de Arrosadia, 31006 Pamplona\\ Spain}
\email{fernando.albiac@unavarra.es}
\author[J. L. Ansorena]{Jos\'e L. Ansorena}
\address{Department of Mathematics and Computer Sciences\\
Universidad de La Rioja\\
Logro\~no\\
26004 Spain}
\email{joseluis.ansorena@unirioja.es}
\author[P. Wald]{Pietro Wald}
\address{Warwick Mathematics Institute, University of Warwick,
Coventry
CV4 7AL
United Kingdom
}
\email{Pietro.Wald@warwick.ac.uk}
\address{Department of Mathematics and Statistics,
University of Jyväskylä,
P.O. Box 35 FI-40014 Finland}
\email{pietro.p.wald@jyu.fi}
\begin{document}
\title[Differentiability of Lipschitz curves in $p$-Banach spaces]{Differentiability of Lipschitz curves in $\bm p$-Banach spaces}
\begin{abstract}
The problem of differentiating Lipschitz curves in Banach spaces goes back to Tamarkin in the early 1930s and is one of the origins of the Radon--Nikod\'ym property in classical Banach space theory. In this article we characterize those quasi-Banach spaces $X$ for which every Lipschitz curve $F\tp\RR\to X$ admits a point of differentiability and prove that this property holds if and only if $X$ is isomorphic to a Banach space with the Radon--Nikod\'ym property, thus substantiating a conjecture of Kalton \cite{Kalton2008}. Equivalently, every nonlocally convex quasi-Banach space admits a nowhere differentiable Lipschitz curve. Our result is obtained from a quantitative characterization of local convexity in terms of the infinitesimal oscillation of Lipschitz curves. Motivated by this obstruction, we then investigate positive differentiability phenomena in $\ell_p$, $0<p<1$, and identify natural decoupled-coordinate constructions for which Lipschitz regularity nevertheless implies almost-everywhere differentiability. We also contrast this behavior with metric differentiability for the natural metric and the quasi-metric on $\ell_p$.
\end{abstract}
\thanks{F.\@ Albiac acknowledges the support of the Spanish Ministry for Science, Innovation, and Universities under Grant PID2025-167660NB-I00 funded by MICIU/AEI/10.13039/501100011033 and ERDF/EU. F.\@ Albiac and J.\@ L.\@ Ansorena acknowledge the support of the Spanish Ministry for Science, Innovation, and Universities under Grant PGC2018-095366-B-I00 for \emph{Functional Analysis Techniques in Approximation Theory and Applications (TAFPAA)}. P.\@ Wald was supported by the Warwick Mathematics Institute Centre for Doctoral Training, and acknowledges funding from the University of Warwick and the UK Engineering and Physical Sciences Research Council (Grant number: EP/W524645/1). He acknowledges funding through the projects \emph{GeoQuantAM: Geometric and Quantitative methods in Analysis on Metric spaces} and \emph{Quantitative differentiability and rectifiability in metric spaces} when ultimating the latest revision.}
\maketitle
\section{Introduction}\noindent
Differentiability is one of the fundamental tools by which nonlinear information can be converted into linear structure. If a Lipschitz mapping between normed spaces is differentiable at a point, its derivative provides a bounded linear operator which retains much of the metric information carried by the original map. This principle has played an important role in the nonlinear classification of Banach spaces and, in particular, in results which extract linear embeddings or isomorphisms from Lipschitz equivalences (see, for instance, \cite{HeinrichMankiewicz1982}, \cite{BenLin}*{Chapter 5} and \cite{AlbiacKalton2016}*{Chapter 14}).

It is therefore natural to ask how far differentiation methods extend beyond the locally convex setting. Quasi-Banach spaces provide the natural framework in which to pose this question. They retain a rich linear and metric structure, while the failure of local convexity deprives us of several of the basic tools upon which classical differentiation theory depends. The purpose of the present paper is to show that, for Lipschitz curves, this loss of local convexity has a remarkably rigid consequence.

The starting point goes back to one of the earliest questions in vector-valued analysis. In the early 1930s, Tamarkin asked for which Banach spaces $X$ every Lipschitz mapping $F\tp[0,1]\to X$ is differentiable almost everywhere. This apparently elementary one-dimensional problem became one of the sources of the theory that is now associated with the Radon--Nikod\'ym property.

The historical progress that followed was rapid. Bochner's work on vector-valued integration \cites{Bochner1933,Bochner1933b} connected differentiation of Banach-space-valued functions with the possibility of recovering absolutely continuous mappings from their derivatives. Birkhoff \cite{Birkhoff1935} proved that Hilbert spaces have the differentiability property considered by Tamarkin. Shortly afterwards, Clarkson \cite{Clar} showed that uniformly convex Banach spaces enjoy the same property and observed that $\ell_1$ does as well, whereas $c_0$ and $L_1[0,1]$ do not. In the same period, Dunford and Morse \cite{DunMor} proved that every Banach space with a boundedly complete basis has the Gelfand--Fr\'echet property, as this differentiability property was sometimes called.

These developments eventually became part of the general theory of the Radon--Nikod\'ym property. In modern terminology, a Banach space $X$ has the Radon--Nikod\'ym property (RNP) if and only if every Lipschitz curve $ F\tp\RR\to X$ is differentiable almost everywhere. This condition is equivalent to a long list of fundamental properties involving vector measures, Bochner integration, martingale convergence, dentability, and representations of operators defined on $L_1$. Thus, the differentiability of one-dimensional Lipschitz curves is not merely one manifestation of the RNP among many others: historically and conceptually, it lies at the origin of the property itself.

There is a further reason for revisiting Tamarkin's question in the quasi-Banach setting. Suppose that $\Phi\tp Y\to X$ is a Lipschitz map between quasi-normed spaces, i.e.,
\[
\Vert \Phi(u)-\Phi(v)\Vert _X
\leq C \Vert u-v\Vert _Y, \qquad u,v\in Y,
\]
for some $C\in[0,\infty)$. Assume, as it is customary, that $\norm{\cdot}_X$ is a continuous map. Then, if the derivative $D(\Phi)(y)$ exists at some $y\in Y$,
\[
\Vert D(\Phi)(y)z\Vert _X \leq C\Vert z\Vert _X,\quad z\in Y.
\]
If $\Phi$ is bi-Lipschitz, i.e., there is a further constant $c>0$ such that
\[
c \Vert u-v\Vert _Y
\leq\Vert \Phi(u)-\Phi(v)\Vert _X, \qquad u,v\in Y,
\]
then
\[
\Vert D(\Phi)(y)z\Vert _X \ge c\Vert z\Vert _Y, \qquad \text{for all}\, z\in Y.
\]
Hence $D(\Phi)(y)$ is a linear isomorphic embedding of $Y$ into $X$. Before differentiation can be used as a linearization tool for nonlinear mappings between quasi-Banach spaces, however, one must first settle the most elementary directional problem: should a Lipschitz mapping from the real line into a quasi-Banach space possess a point of differentiability at all?

Throughout this paper, $I$ will denote an interval of the real line. Let $X$ be a quasi-Banach space. A map $F\tp I \to X$ is said to be Lipschitz if
\[
\Vert F(t)-F(s)\Vert _X\leq L|t-s|, \qquad s,t\in I,
\]
for some $L<\infty$, and $F$ is differentiable at $s\in I$ if there is $v\in X$ such that
\[
\lim_{t\to s} \norm{\frac{F(t)-F(s)}{t-s}-v} _X=0.
\]
Thus, differentiability is always understood in the strong sense, with convergence taking place in the quasi-norm topology.

At first sight, there are compelling reasons to expect a theory radically different from the Banach-space one. Let $0<p<1$ and consider
\begin{equation}\label{eq:LipLp}
F\tp [0,1]\to L_p[0,1],
\qquad
F(t)=\chi_{[0,t]}.
\end{equation}
Then
\[
\Vert F(t)-F(s)\Vert _p=|t-s|^{1/p},
\]
and consequently
\[
\lim_{t\to s} \frac{\Vert F(t)-F(s)\Vert _p}{|t-s|}
=\lim_{t\to s} |t-s|^{1/p-1}= 0.
\]
Thus $F$ is a nonconstant Lipschitz curve whose derivative is identically zero. More generally, Kalton\cite{KaltonZero} proved the following result.

\begin{theorem}[\cite{KaltonZero}*{Theorem 3.3}]\label{thm:KZ}
Suppose that a quasi-Banach space $X$ has trivial dual. Then for every $x\in X$ there is a differentiable curve $F\tp[0,1]\to X$ such that $F^\prime=0$, $F(0)=0$, and $F(1)=x$.
\end{theorem}

Thus, even when differentiation exists, the derivative need no longer determine the curve. The situation appears at first sight substantially more promising for spaces with the point separation property, that is, quasi-Banach spaces $X$ such that for every $x\in X$ there is $x^*$ in the dual space $X^*$ with $x^*(x)\not =0$. Classical examples are the sequence spaces $\ell_p$ for $0<p<1$. Their canonical basis is boundedly complete, while their Banach envelope is $\ell_1$, a space with the RNP. These facts provide rather compelling evidence that $\ell_p$ might retain at least part of the classical differentiability theory, a question that was investigated in \cite{AlbiacAnso2016} and where the following result was proved. Let us recall some terminology that we will use.

A map $F\tp I\to X$ is \emph{weakly differentiable} at $t\in I$ if there is $x\in X$ such that for every $x^*\in X^*$, the scalar function $x^*\circ F$ is differentiable at $t$, and $(x^*\circ F)^\prime=x^*(x)$. Needless to say, (strong) differentiability implies weak differentiability. We say that a quasi-Banach space has the \emph{weak Radon--Nikod\'ym property} (WRNP) is every Lipschitz curve $F\tp I \to X$ is almost everywhere weakly differentiable. Finally, throughout this paper, the word \emph{basis} will always refer to a Schauder basis.

\begin{theorem}[see \cite{AlbiacAnso2016}*{Theorem 4.5}]\label{thm:AA}
Let $X$ be a quasi-Banach space with a boundedly complete basis $(e_n)_{n=1}^\infty$. Let $\widehat{X}$ the Banach envelope of $X$, and $J_X\tp X\to \widehat{X}$ be the envelope map. Then $X$ has the WRNP. Moreover, if $f\tp I \to X$ is the a.e.\@ weak derivate of a Lipschitz curve $F\tp I\to X$, and
$(F_n)_{n=1}^\infty$ are the coordinate functions of $F$, i.e,
\[
F(t)=\sum_{n=1}^{\infty}F_n(t)e_n, \quad t\in I,
\]
then $f(s)=\sum_{n=1}^{\infty}F_n'(s) e_n$, and
\[
\lim_{t\to s} \norm{ J_X\enpar{\frac{F(t)-F(s)}{t-s} -f(s)}}_{\widehat{X}}=0.
\]
a.e\@ $s\in I$.
\end{theorem}

Applying Theorem~\ref{thm:AA} with $X=\ell_p$, $0<p<1$, gives that, for every Lipschitz map $F\tp I \to \ell_p$ there is $f\tp I \to \ell_p$ such that
\[
\lim_{t\to s} \norm{\frac{F(t)-F(s)}{t-s} -f(s)}_1=0.
\]
What could not be concluded was that
\[
\lim_{t\to s}\norm{\frac{F(t)-F(s)}{t-s} -f(s)}_p=0
\]
even at a single point $s\in I$. In the Banach-space argument of Dunford and Morse, the corresponding step in the proof ultimately rests upon Bochner integration, and no analogous mechanism is available in the nonlocally convex setting (see \cites{AlbiacAnsorenaJFA1, AlbiacAnsorenaJFA2}). The problem is therefore not simply one of identifying the candidate derivative: in spaces such as $\ell_p$ the candidate already exists almost everywhere. The real issue is whether the difference quotients converge to it in the original quasi-Banach topology.

These digressions naturally lead to the extension of Tamarkin's question to quasi-Banach spaces. Although this question was implicit in the literature, it was explicitly raised in \cite{AlbiacAnso2016}*{Problem~4.1}.

\begin{questionintro}[Tamarkin's question for quasi-Banach spaces]
Which quasi-Banach spaces $X$ have the property that every Lipschitz curve $F\tp\RR\to X$
is differentiable almost everywhere?
\end{questionintro}

One may ask an apparently much weaker question. Namely, which quasi-Banach spaces have the property that every Lipschitz curve possesses at least one point of differentiability? One of the main consequences of our results is that these two properties are, in fact, equivalent, and that they characterize precisely those quasi-Banach spaces which are isomorphic to Banach spaces with the RNP.

\begin{theoremintro}\label{thm:RNP-characterization}
Let $X$ be a quasi-Banach space. The following conditions are equivalent.
\begin{enumerate}[label=(\roman*),leftmargin=*,widest=iii]
\item\label{it:RNP1} Every Lipschitz map $F\tp\RR\to X$ is differentiable almost everywhere.
\item\label{it:RNP2} Every Lipschitz map $F\tp\RR\to X$ is differentiable at some point.
\item\label{it:RNP3} $X$ is locally convex and has the RNP.
\end{enumerate}
\end{theoremintro}

Perhaps surprisingly, although the original differentiability formulation of the Radon--Nikod\'ym property admits an exact extension to the quasi-Banach category, this extension yields no nonlocally convex examples. The existence of merely one point of differentiability for every Lipschitz curve already forces the underlying space to be locally convex, and, after renorming, what remains is exactly the classical Banach-space RNP. This result is consistent with the theme already hinted at in \cite{AlbiacJOCA} that seemingly natural Banach-space properties of Lipschitz maps often force local convexity when transplanted to quasi-Banach spaces.

Theorem~\ref{thm:RNP-characterization} follows from a more quantitative result which identifies local convexity directly through the infinitesimal oscillation of Lipschitz curves. For a Lipschitz map $\Phi\tp M\to N$ between quasi-metric spaces and a limit point $y\in M$, put
\[
\lip(\Phi;y)=\liminf_{r\to 0^+} \sup_{x\in B(y,r)} \frac{d_N(\Phi(x),\Phi(y))}{r}
\]
and
\[
\Lip(\Phi;y)=\limsup_{r\to 0^+} \sup_{x\in B(y,r)} \frac{d_N(\Phi(x),\Phi(y))}{r}.
\]
These quantities measure, respectively, the lower and upper infinitesimal Lipschitz behavior of $\Phi$ at $y$. Our underlying geometric result is the following characterization of local convexity.

\begin{theoremintro}\label{thm:local-convexity}
Given a quasi-Banach space $X$, the following conditions are equivalent.
\begin{enumerate}[label=(\roman*),leftmargin=*,widest=iii]
\item\label{it:TLC1} $X$ is locally convex.
\item\label{it:TLC2} There exists $0<\delta\le 1$ such that for every Lipschitz map $F\tp\RR \to X$ there is $s\in\RR$ such that
\[
\delta \Lip(F;s)\leq \lip(F;s).
\]
\item\label{it:TLC4} There is $0<\delta\le 1$ such that for every Lipschitz map
$F\tp\RR \to X$ there is $s\in\RR$ such that
\[
\delta\limsup_{t\to s}\frac{\| F(t)-F(s)\|}{|t-s|}\leq\liminf_{t\to s} \frac{\|F(t)-F(s)\|}{|t-s|}.
\]
\item\label{it:TLC3} There is $0<\delta \le 1$ such that for every Lipschitz map $F\tp\RR\to X$,
\[
\delta\limsup_{t\to s}\frac{\| F(t)-F(s)\|}{|t-s|}\leq\liminf_{t\to s}\frac{\|F(t)-F(s)\|}{|t-s|}
\]
for
almost every $s\in\RR$.
\end{enumerate}
\end{theoremintro}

The quantitative content of Theorem~\ref{thm:local-convexity} is worth emphasizing. If $X$ is not locally convex, then, for every $0<\delta<1$, there exists a Lipschitz curve $F\tp\RR\to X$ such that
\[
\lip(F;t)\leq\delta\Lip(F;t), \qquad t\in\RR.
\]
In fact, the construction yields two sequences of scales decreasing to zero at which the normalized oscillations of $F$ remain uniformly separated. Thus the failure of differentiability does not arise from an exceptional set or from isolated irregularities: incompatible first-order behavior is built into the curve at every point and at arbitrarily small scales.

Our results also clarify the role of differentiability in nonlinear classification. In Banach-space theory, the RNP supplies a general mechanism for obtaining linear tangents to Lipschitz mappings. Theorem~\ref{thm:RNP-characterization} shows that no analogous target-space principle based on differentiability can exist in a genuinely nonlocally convex quasi-Banach space. A particular Lipschitz or bi-Lipschitz mapping may, of course, be differentiable, and whenever it is, its derivative remains an effective linearization tool. What fails is the structural guarantee that a point of differentiation must exist. Nonlinear classification in the nonlocally convex category must therefore either impose additional hypotheses on the mappings under consideration or make use of different linearization procedures.

Theorem~\ref{thm:RNP-characterization} also places earlier attempts to extend the RNP to quasi-Banach spaces in perspective. Kalton introduced in \cite{KaltonRNP} the notion of $p$-triviality, motivated by operator-theoretic properties of $L_p$ for $p<1$, while other approaches have emphasized vector measures, martingales, weak differentiation, or alternative notions of integration. While these notions capture meaningful phenomena specific to nonlocally convex spaces, the original analytic formulation of the RNP brings the problem back exactly to the locally convex setting. Moreover, as we will show, the fact that the Banach envelope has the RNP does not guarantee the existence of weak derivatives of Lipschitz maps.

Theorem~\ref{thm:RNP-characterization} settles the general problem of differentiability in a decisive way: outside the locally convex setting there can be no universal differentiability theorem for Lipschitz curves, not even one guaranteeing a single point of differentiability. This negative conclusion makes it natural to turn to a more refined question, namely, which additional assumptions on a Lipschitz curve with values in a nonlocally convex quasi-Banach space ensure that differentiability is nevertheless recovered almost everywhere. The spaces $\ell_p$, $0<p<1$, provide a particularly natural setting in which to pursue this question. By Theorem~\ref{thm:AA}, every Lipschitz curve $F\tp I\to\ell_p$ has a well-defined weak derivative almost everywhere. Thus, after Theorem~\ref{thm:RNP-characterization}, the relevant problem is no longer whether differentiability must always occur, but rather to identify natural classes of curves for which the almost-everywhere weak derivative is automatically promoted to a strong derivative. This motivates the study of decoupled-coordinate mappings into $\ell_p$, $0<p<1$, of the form
\[
F(t)=\sum_{n=1}^{\infty}a_n\, \phi(b_nt)\,e_n, \qquad t\in\RR,
\]
where $\phi\tp\RR\to\RR$ is Lipschitz and the parameters $a_n$ and $b_n$ determine the size and frequency of the individual coordinates. Such curves constitute a natural and flexible class of highly oscillatory mappings. Before obtaining Theorem~\ref{thm:RNP-characterization}, this class of curves arose as plausible candidates for constructing nowhere differentiable Lipschitz curves into $\ell_p$. While Theorem~\ref{thm:RNP-characterization} shows that nowhere differentiable curves certainly exist, the point is instead to show that they cannot, in general, be obtained from the most natural decoupled-coordinate constructions. Thus, although the geometry of $\ell_p$ permits Lipschitz curves with no points of differentiability, the pathology necessarily involves a more intricate interaction between coordinates than that present in this standard oscillatory model.

There is a further distinction which is peculiar to the nonlocally convex setting. If $X$ is a $p$-Banach space, $0<p<1$, then the mapping
\begin{equation*}
(x,y) \mapsto \Vert x-y\Vert_X^p
\end{equation*}
defines a
translation-invariant metric compatible with its topology. Lipschitz regularity with respect to this metric, however, is fundamentally different from that with respect to the quasi-norm Lipschitz condition considered above. In the case where $X=\ell_p$ or, more generally, $X$ is a $p$-Banach space with the point separation property, Lipschitz curves relative to the metric $\norm{\cdot}_X^p$ are constant. This provides a sharp contrast with quasi-norm differentiability, where nonconstant Lipschitz curves abound and the existence of derivatives reflects subtle features of the underlying linear geometry.

The paper is organized as follows. In Section~\ref{sec:Background} we recall the necessary background on quasi-Banach spaces, Banach envelopes, and weak, strong, and metric differentiability of Lipschitz curves. We also derive some useful sufficient conditions for almost-everywhere differentiability in $\ell_p$. Section~\ref{sec:Tamarkin} contains the main results of the paper. In it, we construct, in every nonlocally convex quasi-Banach space, Lipschitz curves with incompatible infinitesimal behavior at arbitrarily small scales, prove the characterization of local convexity in Theorem~\ref{thm:local-convexity}, and deduce Theorem~\ref{thm:RNP-characterization}, thereby settling Tamarkin's problem in the quasi-Banach category. In Section~\ref{sec:new}, we revisit the results from \cite{AlbiacAnso2016} concerning the connection between differentiability of Lipschitz curves and Vogt integrability, and advance the state-of-art of the subject by characterizing Vogt integrable functions taking values in $\ell_p$. This characterization delimits the scope of Vogt integrability for constructing Lipschitz functions into $\ell_p$. In Section~\ref{sec:weakRNP}, we provide examples that illustrate the obstructions that can arise when we attempt to transfer RNP-like properties from the Banach envelope of a quasi-Banach space to the space itself. In Section~\ref{sec:decoupled} we return to $\ell_p$, $0<p<1$, and analyze decoupled-coordinate constructions, showing that a broad class of such Lipschitz mappings is nevertheless differentiable almost everywhere. Finally, Section~\ref{sec:metric} compares the preceding theory with metric differentiability relative to the snowflaking $\norm{\cdot}_X^p$, and proves the rigidity of Lipschitz curves in this metric setting.
\section{Preliminaries}\label{sec:Background}\noindent
For background on quasi-Banach spaces we refer the reader to \cite{KPR1985}. Let us recall first that a \emph{quasi-norm} on a vector space $X$ over the real or complex field $\FF$ is map $\norm{ \cdot}_X\tp X\to [0, \infty)$ with the following properties.
\begin{enumerate}[label=(\roman*)]
\item\label{QN:a} $\norm{ x}_X >0$ for all $x\not=0$,
\item\label{QN:b} $\norm{\alpha\, x}_X=\abs{\alpha} \norm {x}_X$ for all $\alpha\in\RR$ and all $x\in X$.
\item\label{QN:c} There is a constant $\kappa\ge 1$ so that for all $x$ and $y\in X$ we have
\begin{equation*}
\norm{x+y}_X \le \kappa\enpar{\norm{x}_X +\norm{ y}_X}.
\end{equation*}
\end{enumerate}
The smallest constant $\kappa$ in \ref{QN:c} is called the \emph{modulus of concavity} of the quasi-norm. In general, for $m\in\NN$, we denote by $\kappa_m(X)$ the smallest constant $C$ such that
\begin{equation}\label{kappam}
\norm{\sum_{j=1}^m x_j}_X \le C \sum_{j=1}^m \norm{x_j}_X, \quad x_j\in X.
\end{equation}

A quasi-norm $\norm{\cdot}_X$ induces a metric linear topology on $X$. We call these topological vector spaces \emph{quasi-normed spaces} (normed if $\kappa=1$) or, in the case when they are complete, \emph{quasi-Banach spaces} (Banach spaces if $\kappa=1$). The symbol $B_X$ stands for the closed unit ball of $X$.

Given $0<p\le 1$, the quasi-norm $\norm{\cdot}_X$ is said to be a \emph{$p$-norm} if it satisfies \ref{QN:a}, \ref{QN:b} and it is $p$-subadditive, i.e.,
\begin{enumerate}[label=(\roman*),resume]
\item\label{QN:d} $\norm{x+y}_X^p\le \norm{ x}_X^p +\norm{y}_X^p$ for all $x$, $y\in X$.
\end{enumerate}

Of course, \ref{QN:d} implies \ref{QN:c} with $\kappa=2^{1/p-1}$. A $p$-normed (resp., $p$-Banach) space is a quasi-normed (resp., quasi-Banach) space equipped with a $p$-norm. Unless otherwise stated, we regard quasi-normed spaces as quasi-metric spaces equipped with the quasi-distance
\[
(x,y)\mapsto \norm{x-y}_X.
\]

A subset $B$ of a vector space $X$ is said to be absolutely $p$-convex if $\lambda x + \mu y\in B$ for all $x$, $y\in B$ and all $\lambda$, $\mu\in\RR$ with $\abs{\lambda}^p + \abs{\mu}^p \le 1$. We say that the quasi-normed space $X$ is locally $p$-convex (locally convex if $p=1$) if the origin has a basis of neighbourhoods consisting of absolutely $p$-convex sets. The following conditions on $X$ are equivalent.
\begin{enumerate}[label=(\alph*)]
\item $X$ is locally $p$-convex.
\item\label{it:Anso} There is a constant $C\in[1,\infty)$ such that
\[
\norm{\sum_{j\in J} x_j}_X \le C \enpar{\sum_{j\in J} \norm{x_j}_X^p}^{1/p}
\]
for all finite families $(x_j)_{j\in J}$ in $X$.
\item $X$ can be equipped with an equivalent $p$-norm, i.e., it is isomorphic to a $p$-Banach space.
\end{enumerate}

We call the optimal constant $C$ in \ref{it:Anso} the \emph{$p$-convexity constant} (convexity constant if $p=1$) of $X$. Equivalently, the $p$-convexity constant of $X$ is is the smallest constant $C$ such that
\[
\norm{\sum_{j\in J} \lambda_j x_j}_X \le C \enpar{\sum_{j\in J} \lambda_j^p}^{1/p} \sup_{j\in J} \norm{x_j}_X
\]
for all finite families $(x_j)_{j\in J}$ in $X$ and all $(\lambda_j)_{j\in J}$ in $[0,\infty)$.

Let us record the following obvious lemma for further reference.

\begin{lemma}\label{lem:Anso}
Let $X$ be a locally convex quasi-normed space with convexity constant $C$. Let $(t_j)_{j=0}^m$ be a a partition of an interval $[a,b]$. Then, given $F\tp [a,b] \to X$, we have
\[
\norm{F(b)-F(a)}_X \le C (b-a) \sup_{1\le j \le m} \norm{f(t_j)-f(t_{j-1})}.
\]
\end{lemma}

By the Aoki--Rolewicz theorem (see e.g. \cite{KPR1985}) any quasi-normed space is locally $p$-convex for some $0<p\le 1$. Consequently, for any quasi-Banach space $X$ and any $0<\alpha<1$ there is a constant $C_\alpha(X)\in[1,\infty)$ such that the series $\sum_{j=1}^\infty \alpha^{j-1} x_j$ converges, and
\begin{equation}\label{eq:ExpGalb}
\norm{\sum_{j=1}^\infty \alpha^{j} x_j}_X \le C_\alpha(X) \sup_{j\in J} \norm{x_j}_X
\end{equation}
for all bounded sequences $(x_j)_{j=1}^\infty$ in $X$. Using the terminology introduced by Turpin in his study of convexity in topological vector spaces (see \cite{Turpin1976}), this means that quasi-Banach spaces are exponentially galbed.

\begin{lemma}\label{lem:ExpGalb}
Let $X$ be a quasi-Banach space. Then
\[\lim_{\alpha\to 0^+} C_\alpha(X)=0.\]
\end{lemma}
\begin{proof}
Since $\alpha^j \le \alpha \beta^{j-1}$ for all $j\in\NN$ and $0<\alpha\le\beta<\infty$,
\[
C_\alpha(X) \le \frac{\alpha}{\beta} C_\beta(X), \quad 0<\alpha\le \beta <1.\qedhere
\]
\end{proof}

Another consequence of the Aoki--Rolewicz theorem is that, despite the existence of quasi-norms that are discontinuous maps relative to the topology they induce (see \cite{Hyers1939}), any quasi-normed space can be equipped with a continuous quasi-norm. Indeed, $p$-norms are always continuous by the reverse $p$-triangle law.

The \emph{Banach envelope} of a quasi-Banach space $X$ consists of a Banach space $\widehat{X}$ together with a linear contraction $J_X\tp X \to \widehat{X}$ satisfying the following property: for every Banach space $Y$ and every continuous linear map $T\tp X \to Y$ there is a unique continuous linear map $\widehat{T}\tp\widehat{X}\to Y$ such that $\widehat{T}\circ J_X=T$, that is,
\[
\xymatrix{
\widehat{X}\ar[drr]^{\widehat{T}} & \\
X \ar[u]^{J_X} \ar[rr]_T && Y
}
\]
and the ``extension'' $\widehat{T}$ has a norm bounded by the norm of $T$. In particular, $X$ and $\widehat{X}$ have the same dual space.

The \emph{(Banach) envelope map} $J_X$ is one-to-one if and only if $X$ has the point separation property.

Given a function $F\tp I \to X$ we put
\[
\Lip(F)=\sup \enbrace{ \frac{\norm{F(t)-F(s)}_X}{\abs{t-s}} \tq t,s \in I, \, t\not=s}.
\]
Hence, $F$ is Lipschitz if and only if $\Lip(F)<\infty$.

While the information we are used to obtaining about a Lipschitz function $F$ from its derivative may be distorted when $F$ maps into a quasi-Banach space without the point separation property (see Theorem~\ref{thm:KZ}), weak derivatives determine those functions mapping into quasi-Banach spaces with the point separation property. Indeed, suppose that $F$ has zero weak derivative at almost every $t\in I$ and there is $a\in I$ with $F(a)=0$. Then, $x^{\ast}\circ F = 0$ for all $x^{\ast}\in X^{\ast}$. In the case when $X$ has the point separation property, we obtain $F=0$.

Concerning Theorem~\ref{thm:AA}, we recall that a basis $(e_n)_{n=1}^{\infty}$ for a quasi-Banach space $X$
is \emph{boundedly complete} if whenever $(a_n)_{n=1}^{\infty}$ is a sequence of scalars such that
\[
\sup_{m}\norm{\sum_{n=1}^{m} a_n\, e_n}<\infty,
\]
then the series $\sum_{n=1}^{\infty} a_{n}\, e_{n}$ converges in $X$.

Nonlocally convex spaces with a boundedly complete basis display a striking gap between weak and strong differentiability. The last observation in this section shows instances when the remaining obstruction to attaining strong differentiability at a point of a Lipschitz curve can be described very precisely. To that end we need to import a geometric feature of Banach spaces to the nonlocally convex setting.

A quasi-Banach space $X$ is said to have the \emph{Kadec--Klee property} if $\lim_j x_j=x$ for all sequences $(x_j)_{j=1}^\infty$ in $X$ and $x\in X$ such that $\lim_j x_j=x$ weakly and $\lim_j \norm{x_j}_X=\norm{x}_X$.

Observe that if $x\in X$ is such that $x^*(x)=0$ for all $x^*\in X^*$, then the constant sequence $(x_j)_{j=1}^\infty$ given by $x_j=x$ for all $j\in\NN$ weakly converges to $-x$. Hence, quasi-Banach spaces with the \emph{Kadec--Klee property} have the point separation property.

\begin{proposition}\label{prop:KKproperty}
Let $X$ be a quasi-Banach space with a basis $(e_n)_{n=1}^\infty$. Let $(e_n^*)_{n=1}^\infty$ denote the sequence of biorthogonal functionals. Assume that the quasi-norm $\norm{\cdot}_X$ of $X$ is continuous and that there is $p\in(0,\infty)$ such that
\[
\norm{y+z}_X^p \ge \norm{y}_X^p + \norm{z}_X^p
\]
for all $y$, $z\in X$ disjointly supported. Let $(x_j)_{j=1}^\infty$ be a sequence in $X$ and $x\in X$. If
\[
\lim_j e_n^*(x_j)=e_n^*(x)
\]
for all $n\in\NN$, and
\[
\lim_j \|x_j\|_X = \|x\|_X,
\]
then, $\lim_j x_j=x$ in the quasi-norm topology. In particular, $X$ has the Kadec--Klee property.
\end{proposition}

\begin{proof}
For each $m\in\NN$, let $S_m\tp X \to X$ be the partial-sum projection associated with the basis. Let $S_m^c=\Id_X-S_m$ be the complementary projection. By coordinatewise convergence,
\[
\lim_j S_m(x_j)= S_m(x),
\]
whence
\[
\lim_j \norm{S_m(x_j)}_X=\norm{S_m(x)}_X.
\]
Set $\gamma=\kappa_3(X)$. For all $j\in\NN$ we have
\begin{multline*}
\frac{1}{\gamma} \norm{x_j-x}_X^p
\le \norm{S_m(x_j)-S_m(x)}_X^p+ \norm{S^c_m(x)}_X^p + \norm{S^c_m(x_j)}_X^p\\
\le \norm{S_m(x_j)-S_m(x)}_X^p+ \norm{S^c_m(x)}_X^p+ \norm{x_j}_X^p- \norm{S_m(x_j)}_X^p.
\end{multline*}
Letting $j$ go to infinity and then $m$ go to infinity we obtain
\[
\limsup_j \norm{x_j-x}_X^p
\le \gamma \liminf_m \enpar{ \norm{S^c_m(x)}_X^p+ \norm{x}_X^p- \norm{S_m(x)}_X^p}
=0.\qedhere
\]
\end{proof}

For Lipschitz functions on intervals $I$ mapping into metric (or quasi-metric) spaces, the appropriate notion of first-order differentiability is the metric derivative, introduced by Kirchheim \cite{Kirchheim1994}, which we next recall for the sake of self-reference.

Let $(M,d_M)$ be a quasi-metric space and let $F\tp I\to M$ be a mapping. The \emph{metric derivative} of $F$ at a point $s\in I$ is defined
by
\[
\md(F)(s):=\lim_{t\to s}\frac{d_M(F(t),F(s))}{|t-s|},
\]
if the limit exists. The metric derivative measures the speed of the curve $F$ in the quasi-metric space $(M,d_M)$. It is a scalar quantity, in contrast with the vector-valued derivatives considered until now. A fundamental result due to Kirchheim shows that this notion is well behaved for Lipschitz functions into metric spaces.

\begin{theorem}[\cite{Kirchheim1994}*{Theorem 2}]\label{thm:Kir}
Let $(M,d)$ be a metric space and let $F\tp I\to M$ be a Lipschitz map. Then the metric derivative $\md(F)(t)$ exists for almost every $t\in I$.
\end{theorem}

Thus, metric differentiability is automatic for Lipschitz curves independently of any linear or convex structure on the target metric space. As we will see later on, the behavior of Lipschitz functions mapping into quasi-metric spaces is not so clean.

If the quasi-Banach space $X$ has the \emph{Kadec--Klee property}, the differentiability of a Lipschitz map $F\tp I \to X$ at a point $s\in I$ where the weak derivative $f(s)$ exists is equivalent to the fact that the metric derivative exists, and
\[
\md(F)(s)=\norm{f^\prime(s)}_X.
\]
Thus, once almost everywhere weak differentiation has taken place, as it happens in the case when $X$ has a boundedly complete basis, the passage to strong differentiation is reduced to the convergence of the sizes of the difference quotients. In particular, since the spaces $\ell_p$ for $0<p<\infty$ have the Kadec--Klee property by Proposition~\ref{prop:KKproperty}, we have the following.

\begin{corollary}\label{cor:SufCondStrDiff}
Suppose $0<p<1$. Let $F\tp\RR\to \ell_p$ be a Lipschitz map with weak derivative $f$ almost everywhere. Then $F$ is almost everywhere differentiable if and only if $F$ is almost everywhere metric differentiable, and $\md(F)(s)=\norm{f(s)}_p$ a.e.\@ $s\in\RR$.
\end{corollary}

By Corollary~\ref{cor:SufCondStrDiff}, to prevent differentiability it is necessary to force a persistent discrepancy between the upper and the lower limits of the incremental quotients. In this regard, we note that, given a quasi-metric space $(M,d_M)$, a Lipschitz curve $F\tp I \to M$, and $s\in I$,
\begin{align}
\Lip(F;s) &= \limsup_{t\to s} \frac{d_M(F(t),F(s))}{\abs{t-s}},\label{eq:UOvsUP} \\
\lip(F;s)&\ge \liminf_{t\to s} \frac{d_M(F(t),F(s))}{\abs{t-s}}\label{eq:LOvsLL}.
\end{align}
\section{A solution to Tamarkin's problem for quasi-Banach spaces}\label{sec:Tamarkin}\noindent
The elementary geometry behind the construction below is already visible in $\ell_p$, $0<p<1$.

Fix $m\in\NN$ and consider the polygonal curve $F_m\tp[0,1]\to\ell_p$ defined by
\[
F_m(t)=\sum_{n=1}^{i-1} e_n +\left(t-\frac{i-1}{m}\right)e_i,
\quad
t\in\left[\frac{i-1}{m},\frac{i}{m}\right],\quad 1\le i\le m.
\]
Here, an throughout this paper, we use the convention that a sum from $1$ to $0$ is zero. We have
\[
\norm{F_m(1)-F_m(0)}_p= \norm{\sum_{n=1}^m e_n}_p =m^{1/p-1}
\]
and
\[
\lim_{t\to s}
\frac{\norm{F_m(t)-F_m(s)}_p}{|t-s|}=1,
\quad s\in[0,1].
\]
Thus, although the curves $(F_m)_{m=1}^\infty$ have speed one, their endpoints are arbitrarily far apart. Rescaling, this produces $1$-Lipschitz curves which have very small oscillation on a sufficiently fine scale while retaining a large oscillation on a larger scale. The next lemma shows that this phenomenon occurs not only in $\ell_p$ for $0<p<1$ but in arbitrary nonlocally convex quasi-Banach spaces.

\begin{lemma}[cf.\@ \cite{AlbiacJOCA}*{Proposition~3.2}]
\label{lem:small-scale-curve}
Let $X$ be a quasi-normed space. Then $X$ is not locally convex if and only if for every $\varepsilon>0$ there are $r>0$ and a curve $F\tp[0,1]\to X$ such that
\begin{enumerate}[label=(\roman*)]
\item $F(0)=0$, $\norm{F(1)}_X=1$, $\Lip(F)=1$, and
\item\label{it:SSC:2} $\norm{F(t)-F(s)}_X\le \varepsilon |t-s|$ whenever $ |t-s|\le r$.
\end{enumerate}
\end{lemma}

\begin{proof}
Suppose first that $X$ is locally convex, and let $C$ be its convexity constant. Choose $0<\varepsilon<1/C$. If a curve $F$ as in the statement existed, take a partition $(t_i)_{i=0}^m$ of $[0,1]$ with diamenter at most $r$. Then, by Lemma~\ref{lem:Anso},
\[
1\le C \sup_{1\le i \le m} \frac{\norm{F(t_i)-F(t_{i-1})}_X}{t_i-t_{i-1}} \le C\varepsilon<1,
\]
a contradiction.

Conversely, suppose that $X$ is not locally convex. Let $\kappa\ge1$ be modulus of concavity of the quasi-norm. We may assume that $0<\varepsilon<1/\kappa$. There are vectors $(x_i)_{i=1}^m$ in $B_X$ and positive numbers $(\lambda_i)_{i=1}^m$ such that $\sum_{i=1}^m\lambda_i=1$ and
\[
\norm{\sum_{i=1}^m\lambda_i x_i}_X\ge \frac{\kappa}{\varepsilon}.
\]
Set $t_i=\sum_{j=1}^i\lambda_j$, $0\le i \le m$. Define the polygonal curve $G\tp[0,1]\to X$ by
\[
G(t)=\sum_{n=1}^{i-1} x_n+ (t-t_{i-1})x_i, \quad t\in[t_{i-1},t_i], \quad 1\le i \le m.
\]
We have $L:=\Lip(G)\ge \kappa/\varepsilon$. Further, the suppremum defining $L$ is attained, that is, there are $0\le a<b\le 1$ such that
\[
\norm{G(b)-G(a)}_X = L(b-a).
\]
Set
\[
r=\frac{1}{b-a}\min_{1\le j \le m} \lambda_j.
\]
Let $0\le s\le t \le 1$ be such that $t-s\le r(b-a)$. If there is $j=1$, \dots, $m$ such that $t_{j-1}\le s \le t \le t_j$, then
\[
\norm{F(t)-F(s)}_X = (t-s) \norm{x_j}_X.
\]
Otherwise, there is $j=1$, \dots, $m-1$, such that $t_{j-1}\le s \le t_j \le t \le t_{j+1}$. Hence,
\begin{align*}
\norm{G(t)-G(s)}_X
&=\norm{(t_j-s)x_j + (t-t_j) x_{j+1}}_X \\
&\le \kappa \enpar{(t_j-s) \norm{x_j}_X + (t-t_j)\norm{x_{j+1}}_X}.
\end{align*}
In any case, $\norm{G(t)-G(s)}_X\le \kappa (t-s)$. Consequenly, the function
\[
F\tp [0,1] \to X, \quad t\mapsto \frac{G((1-t)a+tb)-G(a)}{L},
\]
satisfies the desired conditions.
\end{proof}

Rescaling and periodizing the curves provided by Lemma~\ref{lem:small-scale-curve} gives the building blocks that will be used in a subsequent multiscale construction.

\begin{lemma}\label{lem:periodic-building-block}
Let $X$ be a quasi-normed space with modulus of concavity $\kappa$. Then $X$ is not locally convex if and only if for every $0<\varepsilon<1$ there exist $r>0$ and a curve $F\tp\RR\to X$ such that
\begin{enumerate}[label=(\roman*)]
\item\label{it:wald1} $\Lip(F)=1$; $\norm{F}_\infty=\varepsilon$;
\item\label{it:wald2} $\norm{F(t)-F(s)}_X\le \varepsilon |t-s|$ whenever $|t-s|\le r$; and
\item\label{it:wald3} for every $s\in\RR$, there is $t\in B_{\RR}(s,\varepsilon)$ such that
\[
\norm{F(t)-F(s)}_X\ge\frac{\varepsilon}{2\kappa}.
\]
\end{enumerate}
\end{lemma}

\begin{proof}
Suppose first that curves with properties \ref{it:wald1}--\ref{it:wald3} exist for every $0<\varepsilon<1$. If $X$ were locally convex, we fix
\[
0<\varepsilon<\frac{1}{2\kappa C},
\]
where $C$ is the convexity constant of $X$, and pick the corresponding curve $F$. Also pick $s\in\RR$. By \ref{it:wald2}, there is $t\in\RR$ with $|t-s|\le\varepsilon$ and
\[
\norm{F(t)-F(s)}_X\ge\frac{\varepsilon}{2\kappa}.
\]
Choose a partition of the interval with endpoints $s$ and $t$ with diamater at most $r$. Combining Property~\ref{it:wald3} with Lemma~\ref{lem:Anso} then gives
\[
C \varepsilon^2 <\frac{\varepsilon}{2\kappa}
\le \norm{F(t)-F(s)}
\le C \varepsilon |t-s|
\le C \varepsilon^2.
\]
This absurdity evinces that $X$ is not locally convex.

Suppose now that $X$ is not locally convex. Let $F_0\tp[0,1]\to X$ and $r_0>0$ be given by Lemma~\ref{lem:small-scale-curve} for the current value of $\varepsilon$. Denote by
\begin{equation}\label{eq:distance}
d_{\RR}(t,2\ZZ)=\inf_{k\in2\ZZ}|t-k|
\end{equation}
the distance from $t$ to $2\ZZ$, and define
\[
F(t)=\varepsilon\, F_0\left(d_{\RR}\left(\frac{t}{\varepsilon},2\ZZ\right)\right), \qquad t\in\RR.
\]
Then $F$ is a $2\varepsilon$-periodic even function with $F(0)=0$, $\norm{F(\varepsilon)}_X=\varepsilon$, $\Lip(F)=1$, and $\norm{F}_\infty=\varepsilon$. Set $r=\varepsilon r_0$. If $t$, $s\in\RR$ satisfy $|t-s|\le r$, then
\[
\abs{d_{\RR}\left(\frac{t}{\varepsilon},2\ZZ\right)-d_{\RR}\left(\frac{s}{\varepsilon},2\ZZ\right)} \le \frac{ \abs{t-s}}{\varepsilon}\le r_0,
\]
whence, by Condition~\ref{it:SSC:2} in Lemma~\ref{lem:small-scale-curve},
\[
\norm{F(t)-F(s)}
\le \varepsilon^2 \abs{d_{\RR}\left(\frac{t}{\varepsilon},2\ZZ\right)-d_{\RR}\left(\frac{s}{\varepsilon},2\ZZ\right)}
\le \varepsilon \abs{t-s}.
\]

It remains to verify \ref{it:wald2}. By periodicity and parity, it suffices to consider the case when $s\in[0,\varepsilon]$.
Since
\[
\varepsilon=\norm{F(\varepsilon)-F(0)}_X\le \kappa \enpar{ \norm{F(s)-F(\varepsilon)}_X+\norm{F(s)-F(0)}_X},
\]
there is $t\in\{0,\varepsilon\}$ such that
\[
\norm{F(s)-F(t)}_X\ge \frac{\varepsilon}{2\kappa}.
\]
Since $\{0,\varepsilon\}\subset B_{\RR}(s,\varepsilon)$, we are done.
\end{proof}

Lemma~\ref{lem:periodic-building-block} provides the building blocks for an iterative construction of the required non-differentiable Lipschitz functions (Proposition~\ref{prop:multiscale-curve}). Similar iterative procedures have already appeared in the context of Lipschitz differentiation (see the paragraph above \cite{ACP2010}*{Definition 1.14} and \cite{Bate_Structure_of_measures_2015}*{Section 4}) and can be traced back to Assouad's embedding theorem \cite{Assouad1983}, see also \cite{Gromov_metric_structures_1999}*{Appendix B}. Section \cite{Gromov_metric_structures_1999}*{B.6} is particularly relevant, as it illustrates the `independence of scales' property of H\"older functions, which suggested the validity of Theorem~\ref{thm:RNP-characterization} and Theorem~\ref{thm:local-convexity}.

We continue our construction by assembling the building blocks provided by Lemma~\ref{lem:periodic-building-block}
at a sequence of increasingly small scales.

\begin{proposition}\label{prop:multiscale-curve}
Let $X$ be a nonlocally convex quasi-Banach space with modulus of concavity $\kappa$. For every $0<\eta<1$ there exist a Lipschitz curve $F\tp\RR\to X$ and sequences $(r_j)_{j=1}^\infty$ and $(\varepsilon_j)_{j=1}^\infty$ of positive numbers decreasing to zero such that
\[
\sup_{t\in B_{\RR}(s,r_j)}\frac{\norm{F(t)-F(s)}_X}{r_j}\le\eta
\]
and
\[
\sup_{t\in B_{\RR}(s,\varepsilon_j)}\frac{\norm{F(t)-F(s)}_X}{\varepsilon_j}\ge\frac{1-\eta}{2\kappa^3}
\]
for every $j\in\NN$ and every $s\in\RR$.
\end{proposition}

\begin{proof}
Set $\gamma=\kappa_3(X)$ in equation \eqref{kappam}. For each $0<\alpha<1$, let $C_\alpha=C_\alpha(X)$ be as in \eqref{eq:ExpGalb}.
By Lemma~\ref{lem:ExpGalb}, choosing $\alpha$ small enough we get
\[
3 \gamma C_\alpha\le \eta
\]
and
\[
\frac{1}{2\kappa^3} - 3 \gamma C_\alpha\ge \frac{1-\eta}{2\kappa^3}.
\]

We will build $(\varepsilon_j,r_j,F_j)_{j=1}^\infty$, where $F_j\tp\RR\to X$ is a curve, recursively. Put $r_0=1$ and define $\varepsilon_0$ and $F_0$ arbitrarily. Assume that $j\in\NN$ and that $\varepsilon_{j-1}$, $r_{j-1}$ and $F_{j-1}$ are constructed. Then, we choose $r=r_j\in(0,\varepsilon_j)$ and $F=F_j$ as in Lemma~\ref{lem:periodic-building-block} with $\varepsilon=\varepsilon_j=\alpha r_{j-1}$.

Since $r_j\le \alpha r_{j-1}$ for all $j\in\NN$, $(r_j)_{j=1}^\infty$ and $(\varepsilon_j)_{j=1}^\infty$ decrease to zero. By induction, $\varepsilon_j \le \alpha^{j-i} r_i$ for all $0\le i<j$. Consequenlty, the function
\[
F=\sum_{j=1}^\infty F_j \tp\RR \to X
\]
is well-defined. Further, if for $k\ge 0$ we set
\[
G_k=F-F_k, \quad S_k=\sum_{j=1}^k F_j, \quad T_k=F-S_k,
\]
then
\begin{equation}\label{eq:TailEstimate}
\norm{T_k(t)}_X \le C_\alpha r_k, \qquad t\in\RR, \, k\ge 0.
\end{equation}

If $k\in\NN$ and $t$, $s\in\RR$ satisfy $\abs{t-s}\le r_k$, then
\[
\norm{F_j(t)-F_j(s)}_X \le \varepsilon_j \abs{t-s},\qquad \text{for all}\; j=1, \dots, k.
\]
Therefore,
\begin{equation}\label{eq:SumEstimate}
\norm{S_k(t)-S_k(s)}_X \le C_\alpha \abs{t-s},
\end{equation}
for all $k\ge 0$, and all $t,s\in\RR$ with $\abs{t-s}\le r_k.$

Fix $k\in\NN$ and $s$, $t\in\RR$. If $\abs{t-s}\le r_k$, combining \eqref{eq:TailEstimate} and \eqref{eq:SumEstimate} we obtain
\begin{align*}
\norm{F(t)-F(s)}_X
&\le \gamma \enpar{ \norm{S_k(t)-S_k(s)}_X +\norm{T_k(t)}_X + \norm{T_k(s)}_X}\\
&\le \gamma C_\alpha \enpar{ \abs{t-s} +2 r_k}\le 3 \gamma C_\alpha r_k\le \delta r_k.
\end{align*}

Given $s\in \RR$ and $k\in\NN$, pick $t\in\RR$ with $\abs{t-s}\le \varepsilon_k$ and
\[
\norm{F_k(t)-F_k(s)}_X\ge \frac{\varepsilon_k}{2\kappa^2}.
\]
Since $ \varepsilon_k\le r_{k-1}$, combining again \eqref{eq:TailEstimate} and \eqref{eq:SumEstimate} yields
\begin{align*}
\norm{G_k(t) - G_k(s) }_X
&\le \gamma \enpar{ \norm{S_{k-1}(t)-S_{k-1}(s)}_X + \norm{T_k(t)}_X + \norm{T_k(s)}_X}\\
&\le \gamma C_\alpha \enpar{\abs{t-s} + 2 r_k } \le 3 \gamma C_\alpha \varepsilon_k\le \eta \varepsilon_k.
\end{align*}
Summing up,
\begin{align*}
\norm{F(t)-F(s)}_X & \ge \frac{1}{\kappa} \norm{F_k(t)-F_k(s)}_X - \norm{G_k(t) - G_k(s) }_X\\
&\ge \varepsilon_k \enpar{ \frac{1}{2\kappa^3} - 3 \gamma C_\alpha }\ge \frac{(1-\eta) \varepsilon_k}{2\kappa^3}.
\end{align*}

We conclude by proving that $F$ is a Lipschitz function. Since it is bounded, it suffices to prove that $F$ is Lipschitz for small distances. Let $t$, $s\in\RR$ with $\abs{t-s}\le 1$. Choose $k\in\NN$ such that $r_{k}\le \abs{t-s}\le r_{k-1}$. Set $\rho=\kappa_4(X)$ in equation \eqref{kappam}. The expansion
\[
F(t)-F(s)=\enpar{S_{k-1}(t)-S_{k-1}(s)}+\enpar{F_k(t)-F_k(s)} + T_k(t)-T_k(s)
\]
yields, in light of \eqref{eq:TailEstimate}, \eqref{eq:SumEstimate} and the fact that $\Lip(F_k)\le 1$,
\begin{align*}
\norm{F(t)-F(s)}_X &\le \rho \enpar{ C_\alpha \abs{t-s} + \abs{t-s} + 2 C_\alpha r_k }\\
&\le \rho \enpar{3 C_\alpha +1} \abs{t-s}.\qedhere
\end{align*}
\end{proof}

We are now ready to prove our main results announced in the Introduction.

\begin{proof}[Proof of Theorem~\ref{thm:local-convexity}]
Clearly, \ref{it:TLC3} implies \ref{it:TLC4}. By \eqref{eq:UOvsUP} and \eqref{eq:LOvsLL}, \ref{it:TLC4} implies \ref{it:TLC2}. In turn, \ref{it:TLC2} implies \ref{it:TLC1} by Proposition~\ref{prop:multiscale-curve}. Assume that \ref{it:TLC1} holds, that is, $X$ is locally convex. Then there is a norm $\norm{\cdot}_0$ on $X$ and $a$, $b>0$ such that
\[
a \norm{x}_X\le \norm{x}_0\le b\norm{x}_X,
\quad x\in X.
\]
Let $F\tp\RR\to X$ be Lipschitz. Viewed as a curve with values in the metric space induced by $\norm{\cdot}_0$, its metric derivative $\md(F)(s)$ exists for almost every $s\in\RR$ by Theorem~\ref{thm:Kir}. At every such point $s$ we have
\[
\frac{a}{b}\limsup_{t\to s}\frac{\norm{F(t)-F(s)}_X}{|t-s|}
\le \frac{\md(F)(s)}{b}
\le \liminf_{t\to s}\frac{\norm{F(t)-F(s)}_X}{|t-s|}.
\]
Hence condition~\ref{it:TLC3} holds with $\delta=a/b$.
\end{proof}

\begin{proof}[Proof of Theorem~\ref{thm:RNP-characterization}]
\par
If $X$ is a Banach space with the RNP, then every Lipschitz curve $F\tp\RR\to X$ is differentiable almost everywhere. Since this latter property passes to isomorphic spaces, \ref{it:RNP3} implies \ref{it:RNP1}.

\ref{it:RNP1} trivially implies \ref{it:RNP2}. Assume now \ref{it:RNP2}. Let $\norm{\cdot}_0$ be a continuous quasi-norm on $X$ and $a$, $b>0$ such that
\[
a \norm{x}_X\le \norm{x}_0\le b\norm{x}_X,
\quad x\in X.
\]
Given a Lipschitz curve $F\tp\RR\to X$, pick a differentiability point $s\in\RR$. We have
\[
\frac{a}{b}\limsup_{t\to s}\frac{\norm{F(t)-F(s)}_X}{|t-s|}
\le \frac{\norm{F'(s)}_0}{b}
\le \liminf_{t\to s} \frac{\norm{F(t)-F(s)}_X}{|t-s|}.
\]
Then, by Theorem~\ref{thm:local-convexity}, $X$ is locally convex. So, $X$ is isomorphic to a Banach space $Y$. Since the property in \ref{it:RNP2} is invariant under linear isomorphisms, every Lipschitz curve from an interval into $Y$ has a point of differentiability, and hence a point of $\varepsilon$-differentiability for every $\varepsilon>0$. By the classical characterization of the Radon--Nikod\'ym property in terms of $\varepsilon$-differentiability of Lipschitz curves (see \cite{BenLin}*{Theorem~5.21}), $Y$ has the RNP. This proves that \ref{it:RNP1} holds.
\end{proof}

Choose a basis point $t_0$ in the real interval $I$. Then, given a quasi-Banach space $X$, the linear space
\[
\Lip_0(I,X) =\enbrace{ F\tp I\to X \tq F \mbox{ is Lipschitz and }F(t_0)=0}
\]
equipped with the quasi-norm $\Lip(\cdot)$ is a quasi-Banach space. By Theorem~\ref{thm:local-convexity}, the space
\[
\LipD_0(I,X) =\enbrace{ F\in \Lip_0(I,X) \tq F \mbox{ is differentiable a.\ e.\@ on }I}
\]
is proper subspace of $\Lip_0(I,X)$ unless $X$ is locally convex. To prove that it closed, we will use that the linear map
\[
\Dt\tp \LipD_0(I,X) \to L_\infty(I,X), \qquad F\mapsto F^\prime
\]
is a contraction.

\begin{proposition}\label{prop:Dp-closed}
Let $X$ be a quasi-Banach space. Then $\LipD_0(I,X)$ is a closed subspace of $\Lip_0(I,X)$.
\end{proposition}

\begin{proof}
Let $(F_k)_{k=1}^{\infty}$ be a sequence in $\LipD_0(I,X)$ and suppose that
\[
\lim_k \Lip(F_k-F)=0
\]
for some $F\in\Lip_0(I,X)$. Since $(F_k^\prime)_{k=1}^\infty$ is a Cauchy sequence in $L_\infty(I,X)$, there is $f\in L_\infty(I,X)$ such that $\lim_k \norm{F_k^\prime-f}_\infty=0$.

For every $k$, let $E_k\subset I$ be a set of full measure such that $F_k$ is differentiable at every point of $E_k$. Let $E_0$ be a set of full measure such that $\lim_k F_k^\prime=f$ uniformly on $I\setminus E_0$. Then, the set
\[
E=\bigcap_{k=0}^{\infty}E_k.
\]
has full measure in $I$. Fix $s\in E$. Set $\gamma=\kappa_3(X)$. The expansion
\begin{multline*}
D(t):=\frac{F(t)-F(s)}{t-s}-f(s)=\enpar{\frac{F_k(t)-F_k(s)}{t-s}-F_k^\prime(s)} \\
+ \frac{(F(t)-F_k(t))-(F(s)-F_k(s))}{t-s}+\enpar{F_k^\prime(s)-f(s)},
\end{multline*}
gives for each $t\in I$ and $k\in\NN$ the estimate
\[
\frac{\norm{D(t)}_X}{\gamma} \le \norm{\frac{F_k(t)-F_k(s)}{t-s}-F_k^\prime(s)}_X + \Lip(F_k-F)+\norm{F_k^\prime-f}_\infty.
\]
Letting $t$ go to $s$ and then $k$ to infinity gives
\[
\limsup_{t\to s} \norm{D(t)}_X \le \gamma \liminf_k \enpar{\Lip(F_k-F)+\norm{F_k^\prime-f}_\infty}=0.\qedhere
\]
\end{proof}

Let us finish this section with an observation pointing to the difficulty of using metric derivatives to study Lipschitz maps into quasi-metric spaces.

\begin{corollary}
Let $X$ be a quasi-Banach space. Assume that every Lipschitz map $F\tp\RR\to X$ is metric differentiable at some point $s\in\RR$. Then $X$ is locally convex.
\end{corollary}

\begin{proof}
It is immediate from characterization~\ref{it:TLC4} of local convexity in Theorem~\ref{thm:local-convexity}.
\end{proof}
\section{Differentiation and integration of curves in \texorpdfstring{$\ell_{p}$}{}, \texorpdfstring{$0<p<1$}{}}\label{sec:new}\noindent
For Banach spaces, differentiation and integration go hand by hand, and the main drawback in implementing the proof of Dunford and Morse's theorem from \cite{DunMor} to be able to guarantee differentiability in Theorem~\ref{thm:AA} is the fact that Bochner integration does not make sense in non-locally convex spaces. Pettis-integrability is of no help in this task. In fact, a Lipschitz function $F\tp I \to X$, where $X$ is a quasi-Banach space is a.e.\@ weakly differentiable with weak derivative $f\tp I \to X$ if and only if
\[
x^*(F(t))-x^*(F(s))=\int_s^t x^*(f(u))\, du
\]
for all $x^*\in X^*$ and $s$, $t\in I$.

Roughly speaking we could say that we cannot differentiate quasi-Banach valued functions because we do not have a proper way to integrate them. This and other issues related to defining a sensible integral with values in a $p$-normed space have been discussed in the articles \cites{AlbiacAnsorenaJFA1,AlbiacAnsorenaJFA2}.

In \cite{Vogt1967} Vogt introduced a concept of integrability that tries to replace to the Bochner integral when dealing with functions with values in nonlocally convex quasi-Banach spaces. Let $(\Omega,\Sigma,\mu)$ be a measure space, $0<p<1$, and $X$ a $p$-Banach space. A function $f\tp\Omega \to X$ is said to be \emph{integrable in the sense of Vogt}, and we write $\displaystyle f\in L^{1}_{V}(\mu,X)$ (also, $ f\in L^{1}_{V}(I,X)$ when $\mu$ is the Lebesgue measure on a subset $I\subset\RR^d$) if $f$ admits an expression of the following guise
\begin{equation}\label{Vogtdecomposition}
f(t)=\sum_{n=1}^{\infty} x_n \, f_n(t),\qquad \mbox{a.e. } t\in I,
\end{equation}
where $\bm{x}=(x_n)_{n=1}^{\infty}$ in $X$ and $\bm{f}=(f_n)_{n=1}^{\infty}$ in $ L_{1}(\mu,\RR)$ verify the condition
\begin{equation} \label{Vogtcondition}
N(\bm{x},\bm{f})=\sum_{n=1}^{\infty} \Vert x_n\Vert^p \Vert f_n\Vert_1^p <\infty.
\end{equation}
Note that \eqref{Vogtdecomposition} implies a.e.\@ unconditional convergence of the series in \eqref{Vogtcondition}. The space $L^{1}_{V}(\mu,X)$ equipped with the gauge
\[
\Vert f\Vert_{1,V} =\inf\left\{N(\bm{x},\bm{f})^{1/p} \tq \eqref{Vogtdecomposition}
\text{ and } \eqref{Vogtcondition} \text{ hold}\right\}
\]
is a $p$-Banach space. Moreover, for $E\in\Sigma$ the expression
\[
\sum_{n=1}^{\infty} x_n\int_E f_n\, d\mu
\]
does not depend on the decomposition \eqref{Vogtdecomposition} chosen for $f$, and so it is consistent to define the \textit{Vogt integral} of $f$ on $E$ as
\[
\int_E f\, d\mu=\sum_{n=1}^{\infty} x_n\int_E f_n\, d\mu.
\]

\begin{proposition}[\cite{AlbiacAnso2016}*{Proposition 4.4}]\label{prop:AA-strong}
Let $X$ be a quasi-Banach space with the point separation property, and let $F\tp I\to X$ be a Lipschitz mapping. Suppose that $F$ is weakly differentiable almost everywhere with weak derivative $f\tp I\to X$. If $f$ is Vogt-integrable, then $F$ is differentiable almost everywhere and $F'(t)=f(t)$ for a.e.\@ $t\in I$.
\end{proposition}

This result motivates studying conditions that ensure that the weak derivative of a Lipschitz map is Vogt-integrable. If $X$ has a normalized basis $(e_n)_{n=1}^\infty$ with coordinate functionals $(e_n^*)_{n=1}^\infty$, then any function $f\tp I\to X$ has a canonical expansion
\[
f=\sum_{n=1}^\infty e_n^*(f) \, e_n.
\]
Thus, a natural sufficient condition for Vogt-integrability is
\begin{equation*}
\sum_{n=1}^\infty \norm{ e_n^*(f) }_{L_1(I)}^p<\infty.
\end{equation*}
This observation leads to the following consequence of Proposition~\ref{prop:AA-strong}.

\begin{corollary}\label{cor:AA-Vogt}
Let $0<p\le 1$, $X$ be a $p$-Banach space with a basis and
\[
F=\sum_{n=1}^\infty F_n\, e_n \tp I \to X
\]
be a Lipschitz map. If
\begin{equation}\label{eq:CoordinateVogt}
\sum_{n=1}^\infty \|F_n^{\prime}\|_{L^1(I)}^{\,p}<\infty,
\end{equation}
then $F$ is differentiable almost everywhere and $F^\prime=\sum_{n=1}^\infty F_n^{\prime} e_n$.
\end{corollary}

Corollary~\ref{cor:AA-Vogt} shows that differentiability almost everywhere follows from structural assumptions on the coordinates of a Lipschitz map $F\tp I\to \ell_p$. In \cite{AlbiacAnso2016}*{Corollary~4.6} one such assumption is that each coordinate function is monotone. For later use, it is convenient to isolate a slightly more flexible hypothesis. Namely, instead of monotonicity, we impose a quantitative non-degeneracy condition on the derivative of each coordinate, namely that its absolute value does not oscillate too much. This includes monotone coordinates as a special case (on intervals where they are differentiable a.e.\@ with essentially constant sign), but also allows mild oscillations while retaining enough control to deduce Vogt integrability of the formal derivative.

\begin{proposition}
Let $0<p<1$ and let $F=(F_n)_{n=1}^\infty\tp I\to \ell_p$ be a Lipschitz map. Suppose that there exists a constant $C\in[1,\infty)$ such that for every $n$,
\begin{equation*}
M_n:= \esssup_{t\in I} |F_n'(t)| \le C \essinf_{t\in I} |F_n'(t)|.
\end{equation*}
Then $F$ is differentiable almost everywhere on $I$.
\end{proposition}

\begin{proof}
Assume without loss of generality that $I$ is bounded, that $0\in I$, and that $F(0)=0$. There is a zero-measure set $N$ such that $M_n\le |F_n'(t)|$ for all $n\in\NN$ and $t\in I\setminus N$. We have
\[
\enpar{\sum_{n=1}^\infty M_n^p}^{1/p} \le\enpar{ \sum_{n=1}^\infty |F_n'(t)|^p}^{1/p} =\norm{F(t)}_p \le \Lip(F).
\]
In turn, for all $n\in\NN$,
\[
\norm{F_n'}_{L_1(I)} \le \abs{I} \norm{F_n'}_{L_\infty(I)}\le C \abs{I} M_n.
\]
Applying Corollary~\ref{cor:AA-Vogt} puts an end to the proof.
\end{proof}

Corollary~\ref{cor:AA-Vogt} exemplifies the convenience to easily recognize the Vogt integrability of functions in a $p$-Banach space $X$. For $X=\ell_{p}$, Vogt-integrability has a particularly simple exact description.

\begin{proposition}\label{prop:vogt-lp}
Let $0<p<1$, let $I$ be a bounded interval, and let $f=(f_n)_{n=1}^{\infty}\tp I\to \ell_p$ be measurable. Then $f\in L_V^1(I,\ell_p)$ if and only if
\begin{equation}\label{Vogtintcondition}
\sum_{n=1}^{\infty} \|f_n\|_{L_1(I)}^p<\infty .
\end{equation}
Further, $\Vert f\Vert_{1,V}=\sum_{n=1}^{\infty} \|f_n\|_{L_1(I)}^p$.
\end{proposition}

\begin{proof}
If suffices to prove that \eqref{Vogtintcondition} is a necessary condition for $f$ to be Vogt-integrable. Choose a Vogt decomposition
\[
f=\sum_{j=1}^{\infty}x_j \, g_j
\]
where $x_j\in\ell_p$, $g_j\in L_1(I)$, and
\[
\sum_{j=1}^{\infty} \|x_j\|_p^p\,\|f_j\|_{L_1(I)}^p<\infty.
\]
Write
\[
x_j=(x_{j,n})_{n=1}^{\infty}, \quad j\in\NN.
\]
Then, for each $n\in\NN$,
\[
f_n=\sum_{j=1}^{\infty}x_{j,n}\, g_j.
\]
Since the inclusion $\ell_p\subset\ell_1$ is a contraction,
\begin{align*}
\sum_{n=1}^{\infty} \|f_n\|_{L_1(I)}^p&\le \sum_{n=1}^\infty \enpar{ \sum_{j=1}^{\infty} \abs{x_{j,n}} \, \norm{g_j}_{L_1(I)}}^p
\le \sum_{n=1}^\infty \sum_{j=1}^{\infty} \abs{x_{j,n}}^p \norm{g_j}_{L_1(I)}^p\\
&=\sum_{j=1}^{\infty} \norm{g_j}_{L_1(I)}^p \sum_{n=1}^\infty \abs{x_{j,n}}^p=\sum_{j=1}^{\infty} \|x_j\|_p^p\,\|f_j\|_{L_1(I)}^p.\qedhere
\end{align*}
\end{proof}

We close this section by remarking that, in spite of Proposition~\ref{prop:vogt-lp}, condition \eqref{eq:CoordinateVogt} is not necessary for differentiability.

\begin{proposition}[see \cite{AlbiacAnsorena2012}*{Theorem 4.1}]\label{prop:diff-not-vogt}
Let $0<p \le 1$ and $X$ be $p$-Banach space. Assume that $F^\prime\in L_{1,V}([0,1],X)$ for every Lipschitz curve $F\in \LipD([0,1],X)$. Then $X$ is locally convex.
\end{proposition}
\section{WRNP vs.\ RNP in quasi-Banach spaces and their envelopes}\label{sec:weakRNP}\noindent
For Banach spaces, Bochner integration closes the gap between weak and strong differentiation, and consequently the WRNP and the RNP coincide. In the nonlocally convex setting this connection may break down. We show first that the RNP of the Banach envelope does not imply the WRNP of the original quasi-Banach space.

\begin{theorem}\label{thm:RNPvsWRNP}
There is a quasi-Banach space without the WRNP whose Banach envelope has the RNP.
\end{theorem}

Before proving Theorem 5.1, we establish two auxiliary lemmas.

If $\Et=(e_n)_{n=1}^\infty$ is a basis of a quasi-Banach space $X$, then the sequence $(J_X(e_n))_{n=1}^\infty$, where $J_X\tp X \to \widehat{X}$ is the envelope map, is a basis of the Banach envelope $\widehat{X}$ that we call the \emph{envelope basis} $\widehat{\Et}$ of $\Et$.

\begin{lemma}\label{lem:BCNBC}
Given $0<p<1$, there exists a $p$-Banach space with an unconditional basis that fails to be boundedly complete, whose envelope basis is a boundedly complete basis of $\widehat{X}$.
\end{lemma}

\begin{proof}
Let $X$ be the separable part of the weak Lorentz space $\ell_{p,\infty}$. Since $\ell_p\subset X\subsetneq \ell_{p,\infty}\subset\ell_1$, the Banach envelope of $X$ is $\ell_1$ via the inclusion map, and the canonical basis of $X$ fails to be boundedly complete. While Hunt \cite{Hunt1966} proved that $X$ is locally $q$-convex for all $0<q<p$, Kalton \cite{Kalton1980b} proved that it is even locally $p$-convex.
\end{proof}

The $c_0$-criterion for bounded completeness of a basis (see, e.g., \cite{AlbiacKalton2016}*{Theorem 3.3.2}) applies to unconditional bases of quasi-Banach spaces, except that, in the absence of local convexity, the copies of $c_0$ we obtain may not be complemented. Since the Hahn--Banach theorem is unavailable, the proof of this result is not a straightforward generalisation of the Banach case. To avoid straying from our goal, we omit this proof. Instead, we include a proof of the precise part of the $c_0$-criterion we will need. The arguments are fairly standard.

\begin{lemma}\label{lem:coBC}
Given an unconditional basis $\Et=(e_n)_{n=1}^\infty$ of a quasi-Banach space $X$, the following are equivalent.
\begin{enumerate}[label=(\roman*),leftmargin=*,widest=iii]
\item\label{it:c0BC1} $\Et$ fails to be boundedly complete.
\item\label{it:c0BC2} There is a block basic sequence $(x_j)_{j=1}^\infty$ of $\Et$ which is equivalent to the canonical $c_0$-basis.
\end{enumerate}
\end{lemma}

\begin{proof}
Let $(e_n^*)_{n=1}^\infty$ be the biorthogonal functionals of $\Et$. There is $K\in[1,\infty)$ such that
\[
\norm{ \sum_{n=1}^\infty b_n \, e_n^*(f) \, e_n } \le K \sup_{n\in\NN} \abs{b_n} \norm{f}
\]
for all $f\in X$ and all sequences $(b_k)_{k=1}^\infty \in c_{00}$ (see \cite{AABW2021}*{Theorem 2.10}).

Pick a sequence of scalars $(a_n)_{n=1}^\infty$. If $\sum_{n=1}^\infty a_n \, x_n$ does not converge, then, by the Cauchy criterion, there are $\varepsilon>0$ and an increasing sequence $(m_k)_{k=1}^\infty$ in $\NN$ such that the block basic sequence $(x_k)_{k=1}^\infty$ defined as
\[
x_k= \sum_{n=m_{2k-1}}^{m_{2k}} a_n \, e_n, \qquad k\in\NN,
\]
satisfies $\norm{x_k}\ge \varepsilon$ for all $k\in\NN$. Consequently,
\[
\norm{\sum_{k=1}^\infty b_k \, x_k} \ge \frac{\varepsilon}{K} \sup_k \abs{b_k}
\]
for all sequences $(b_k)_{k=1}^\infty\in c_{00}$. If, moreover,
\[
C:=\sup_m\norm{\sum_{n=1}^m a_n \, e_n}<\infty,
\]
then
\[
\norm{\sum_{k=1}^\infty b_k \, x_k} \le C K \sup_k \abs{b_k}.
\]
Thus, \ref{it:c0BC1} implies \ref{it:c0BC2}. Conversely given a block basic sequence $\Xt=(x_k)_{k=1}^\infty$ of $\Et$, there are an increasing sequence $(m_k)_{k=0}^\infty$ in $\ZZ$ with $m_0=0$ and $(a_n)_{n=1}^\infty$ in $\FF$ such that
\[
x_k=\sum_{n=1+m_{k-1}}^{m_k} a_n \, e_n.
\]
Set
\[
D=\sup_{j\in\NN} \norm{\sum_{k=1}^j x_k}.
\]
Since $\sum_{n=1}^{m_j} a_n\, e_n=\sum_{k=1}^j x_k$ for all $j\in\NN$,
\[
C:=\sup_{m\in\NN} \norm{\sum_{n=1}^m a_n \, e_n}\le K D.
\]
If $\Xt$ is equivalent to the canonical $c_0$-basis, then $\sum_{k=1}^\infty x_j$ does not converge, and $D<\infty$. Hence, $\sum_{n=1}^{\infty} a_n\, e_n$ does not converge, and $C<\infty$. Consequently, \ref{it:c0BC2} implies \ref{it:c0BC1}.
\end{proof}

The next lemma is a partial converse of Theorem~\ref{thm:AA}.

\begin{lemma}\label{lem:ConverseAA}
Let $X$ be a quasi-Banach space with an unconditional basis $\Et$. If $X$ has the WRNP, then $\Et$ is boundedly complete.
\end{lemma}

\begin{proof}
Suppose that $\Et$ is not boundedly complete.By Lemma~\ref{lem:coBC}, there is a block basic sequence $(x_j)_{j=1}^\infty$ of $\Et$ equivalent to the canonical $c_0$-basis. This implies that a block series $\sum_{j=1}^\infty a_j \, x_j$ converges if and only if $\lim_j a_j=0$. Moreover, the operator
\[
T\tp c_0 \to X, \qquad (a_j)_{j=1}^\infty \mapsto \sum_{j=1}^\infty a_j \, x_j
\]
is linear and bounded. Let
\[
\phi(t)=d(t,2\ZZ),\qquad t\in\RR,
\]
as in \eqref{eq:distance}, and define the curve in $X$
\[
F(t)=\sum_{j=1}^\infty \frac{1}{j}\phi(jt)x_j, \qquad t\in\RR.
\]
Since $\phi$ is bounded, the coefficient sequence $(j^{-1}\phi(jt))_{j=1}^\infty$ belongs to $c_0$, so $F$ is well defined. Moreover, since $\phi$ is $1$-Lipschitz,
\[
\|F(t)-F(s)\|_X
\leq
\|T\|
\sup_{j\in\NN}
\frac{1}{j}|\phi(jt)-\phi(js)|
\leq
\norm{T} \abs{t-s}
\]
for all $s$, $t\in \RR$. Hence $F$ is Lipschitz.

Let $(e_n^*)_{n=1}^\infty$ be the biorthogonal functionals of $\Et$, and choose an increasing sequence $(n_j)_{j=0}^\infty$ of nonnegative integers, with $n_0=0$, such that
\[
\supp(x_j)\subseteq(n_{j-1},n_j], \qquad j\in\NN.
\]
Suppose towards a contradiction that $F$ is weakly differentiable at some point $s\in\RR$, and let $x\in X$ be its weak derivative.

Fix $j\in\NN$ and choose $n\in(n_{j-1},n_j]$ such that $e_n^*(x_j)\neq 0$. Since
$e_n^*(x_k)=0$ for $k\neq j$,
\[
e_n^*(F(t))= \frac{e_n^*(x_j)}{j}\phi(jt).
\]
The scalar function $e_n^*\circ F$ is differentiable at $s$. Therefore $\phi$ is differentiable at $js$. Since the set of
differentiability points of $\phi$ is $\RR\setminus\ZZ$, we have $js\notin\ZZ$ for every $j\in\NN$, and consequently $s\notin\QQ$.

Furthermore, for every $j\in\NN$ and $n\in(n_{j-1},n_j]$,
\[
e_n^*(x)=(e_n^*\circ F)'(s) =e_n^*(x_j)\phi'(js).
\]
Hence the Schauder expansion of $x$ is
\[
x= \sum_{j=1}^\infty \sum_{n=n_{j-1}+1}^{n_j}e_n^*(x)e_n =
\sum_{j=1}^\infty \phi'(js) \sum_{n=n_{j-1}+1}^{n_j}e_n^*(x_j)e_n
= \sum_{j=1}^\infty \phi'(js)\, x_j.
\]
But $|\phi'(js)|=1$ for all $j\in\NN,$ so the coefficients $(\phi'(js))_{j=1}^\infty$ do not tend to zero. This contradiction shows that $F$ is nowhere weakly differentiable. Hence, $X$ fails the WRNP.
\end{proof}

\begin{proof}[Proof of Theorem~\ref{thm:RNPvsWRNP}]
Choose the quasi-Banach space $X$ with a basis $\Et$ provided by Lemma~\ref{lem:BCNBC}. Since the envelope basis is boundedly complete, $\widehat{X}$ has the RNP. However, $X$ fails to have the WRNP by Lemma~\ref{lem:ConverseAA}.
\end{proof}

We do not know whether the reverse phenomenon of Theorem~\ref{thm:RNPvsWRNP} can occur.

\begin{question}\label{qt:WRNPtoRNP}
Let $X$ be a quasi-Banach space with the WRNP. Does the Banach envelope $\widehat{X}$ have the RNP?
\end{question}

For the class of quasi-Banach spaces covered by Theorem~\ref{thm:AA},
Question~\ref{qt:WRNPtoRNP} has an affirmative answer.

\begin{theorem}\label{thm:bounded-completeness-envelope}
Let $X$ be a quasi-Banach space with a boundedly complete basis $\Et=(e_n)_{n=1}^\infty$. Then the envelope basis $\widehat{\Et}$ is a boundedly complete basis of the Banach envelope $\widehat{X}$.
In particular, $\widehat{X}$ has the RNP.
\end{theorem}

Before proving Theorem~\ref{thm:bounded-completeness-envelope}, we introduce some notation.

Given a basis $\Et=(e_n)_{n=1}^\infty$ of a quasi-Banach space $X$ with biorthogonal functionals $(e_n^*)_{n=1}^\infty$, we denote by
\[
\Ft[\Et] \tp X\to \FF^\NN, \quad f\mapsto (e_n^*(f))_{n=1}^\infty
\]
its \emph{coefficient transform}. Note that $\Ft[\Et]$ is one-to-one. We denote by
\[
\St[\Et]\tp \Ft[\Et](X) \to X, \quad (a_n)_{n=1}^\infty \mapsto \sum_{n=1}^\infty a_n \, e_n
\]
its inverse map.

Let $\omega$ denote the topology of coordinatewise convergence on $\FF^\NN$. If the quasi-norm of $X$ is continuous and $\Et$ is monotone, then $\Et$ is boundedly complete if and only if
\[
B(\Et):=\Ft[\Et](B_X)
\]
is $\omega$-closed. Moreover, $B(\Et)$ is contained in the product space
\[
\prod_{n=1}^\infty \enbrace{\lambda\in\FF \tq |\lambda|\leq \norm{e_n^*}_{X^*}},
\]
which is compact by Tychonoff's theorem. Hence, in case it is closed, $B(\Et)$ is compact in the coordinate topology. Since $\FF^\NN$ with the product topology is metrizable, $B(\Et)$ is in fact compact metrizable.

The notation $\co(V)$ stands for the convex hull of a subset $V$ of a topological vector space, and $\cocl(V)$ for the closed convex hull of $V$.

Given $n\in\NN$, we denote by $\pi_n\tp\FF^\NN\to \FF$ be $n$th coordinate functional associated with the unit vector system of $\FF^\NN$.

\begin{proof}[Proof of Theorem~\ref{thm:bounded-completeness-envelope}]
Throughout the proof, the topology on $X$ will be the one induced by its quasi-norm, the topology on $\widehat{X}$ will be the one induced by its norm, and the topology on $\FF^\NN$ will be the one associated with the topology of pointwise convergence. Passing to an equivalent $p$-norm for some $0<p\le 1$ and then replacing it with the quasi-norm
\[
x \mapsto \sup_{m\in\NN}\|S_mx\|_X, \quad x\in X,
\]
where $S_m$ is the $m$th partial-sum projection associated with $\Et$, we may assume that $X$ is a $p$-Banach space and that $\Et$ is monotone. Since $S_m$, $m\in\NN$, extends to a contraction on $\widehat{X}$, $\widehat{\Et}$ is also monotone. Thus, $B(\Et)$ is compact, and to show that $\widehat{\Et}$ is boundedly complete we must prove that $B(\widehat{\Et})$ is closed. To that end, it suffices to see that
\[
C= \cocl(B(\Et))=B(\widehat{\Et}).
\]
Pick $\alpha=(a_n)_{n=1}^\infty\in C$. There is a regular probability measure on $B(\Et)$ such that
\[
f(\alpha)=\int_{B(\Et)} f\, d\mu
\]
for every continuous linear map $f\tp\FF^\NN\to\FF$. Indeed, $C$ is compact since $\FF^\NN$ is a Fr\'echet space (see \cite{Rudin1991}*{Theorem 3.20}), so we may apply \cite{Rudin1991}*{Theorem 3.28}. In particular,
\begin{equation}\label{tag5.2}
a_n=\int_{B(\Et)}\pi_n\,d\mu, \qquad n\in\NN.
\end{equation}
The function $ \Phi:=J_X\circ\St[\Et]$ maps $B(\Et)$ into $B_{\widehat{X}}$. Further,
\[
\Phi(\beta)=\lim_{m\to\infty} \sum_{n=1}^m \pi_n(\beta) \, J_X(e_n), \quad \beta\in B(\Et).
\]
Thus $\Phi$ is strongly measurable and therefore $\Phi$ is Bochner integrable relative to $\mu$. Set
\begin{equation}\label{eq:newtag}
x=\int_{B(\Et)}\Phi\,d\mu\in B_{\widehat X}.
\end{equation}
Let $(\widehat{e}^{\, *}_n)_{n=1}^\infty$ be the sequence of biorthogonal functionals of $\widehat{\Et}$. We have
\begin{equation}\label{eq:newtag2}
\widehat{e}^{\, *}_n \circ J_X =e_n^*
\end{equation}
and
\[
e_n^*\circ \St[\Et]=\pi_n
\]
for all $n\in\NN$. Therefore, by \eqref{tag5.2} and \eqref{eq:newtag},
\[
\widehat{e}^{\, *}_n(x)
= \int_{B(\Et)} \widehat{e}^{\, *}_n \circ J_X\circ\St[\Et]\,d\mu
=\int_{B(\Et)}\pi_n\,d\mu
= a_n, \qquad n\in\NN.
\]
Hence $\alpha\in B(\widehat{\Et})$. This proves that $C\subset B(\widehat \Et)$.

Conversely, the standard realization of the Banach envelope gives
\[
B_{\widehat X} = \cocl\enpar{J_X(B_X)}.
\]
By \eqref{eq:newtag2}, $\Ft[\widehat{\Et}]\circ J_X=\Ft[\Et]$. Since the coefficient transform $\Ft[\widehat{\Et}]$ is continuous,
\begin{multline*}
B(\widehat{\Et})
=\Ft[\widehat{\Et}] \enpar{\cocl(J_X(B_X))}\\
\subset \cocl\enpar{\Ft[\widehat{\Et}](J_X(B_X))}
=\cocl \enpar{ \Ft[\Et](B_X)}=C.
\end{multline*}

This completes the proof since every Banach space with a boundedly complete basis has the RNP.
\end{proof}
\section{Decoupled-coordinate constructions of differentiable Lipschitz curves in \texorpdfstring{$\ell_{p}$}{} for \texorpdfstring{$0<p<1$}{}.}\label{sec:decoupled}\noindent
Theorem~\ref{thm:RNP-characterization} shows that, for $0<p<1$, there exist Lipschitz curves with values in $\ell_p$ which fail to be differentiable at every point. It is therefore natural to ask how such pathological curves can be constructed and, conversely, which additional structural features of a Lipschitz curve force differentiability almost everywhere.

The results recalled in Section~\ref{sec:new} suggest that coordinate structure can play a decisive role in this question. Every Lipschitz mapping
\[
F=\sum_{n=1}^{\infty}F_n\, e_n
\tp I\to\ell_p
\]
has a well-defined formal derivative
\[
f(t)=\sum_{n=1}^{\infty}F_n'(t)e_n
\]
for almost every $t$, and suitable integrability conditions on the coordinate derivatives imply that this formal derivative is in fact the derivative of $F$ almost everywhere. In particular, \cite{AlbiacAnso2016}*{Corollary~4.6} shows that this happens whenever the coordinate functions $F_n$ are monotone. Thus the existence of nowhere differentiable Lipschitz curves in $\ell_p$ does not preclude rather large and natural classes of mappings from retaining the classical almost-everywhere differentiability behavior.

In this section we investigate this phenomenon for a basic class of decoupled-coordinate constructions and exhibit that this mechanism produces almost everywhere differentiable curves in $\ell_{p}$.
We emphasize that our purpose is not to characterize the class $\LipD_0(I,\ell_p)$ of Lipschitz curves $F\tp I\to\ell_p$ which are differentiable almost everywhere. The results obtained above should rather be viewed as identifying a natural and fairly flexible class of curves for which differentiability can be recovered from the particular structure of their coordinates. Indeed, the preceding sections already exhibit rather different mechanisms leading to almost-everywhere differentiability. Vogt integrability of the weak derivative provides one such mechanism, but Proposition~\ref{prop:diff-not-vogt} shows that it is far from being necessary. On the other hand, Proposition~\ref{prop:Dp-closed} shows that $\LipD_0(I,\ell_p)$ is closed in the Lipschitz quasi-norm, and Proposition~\ref{p:remark33-diff} illustrates that even within the decoupled-coordinate setting the behavior may change substantially once the non-flatness assumption in Theorem~\ref{thm:decoupled-ae-diff} is removed. Thus almost everywhere differentiability in $\ell_p$ appears to arise from several genuinely different sources. In view of this variety, there seems to be little reason to expect a simple coordinate or summability criterion characterizing the whole class $\LipD(I,\ell_p)$. A more intrinsic description of this class, if one is available, would require a substantially broader perspective than the one pursued here.

\begin{theorem}\label{thm:decoupled-ae-diff}
Let $0<p<1$, $(a_n)_{n=1}^\infty$ and $(b_n)_{n=1}^\infty$ be sequences in $(0,\infty)$, and let $\phi\tp\RR\to\RR$ be Lipschitz with $\phi(0)=0$ and $\phi'(0)\not=0$. For $t\in \RR$ consider the formal series
\[
F(t) =\sum_{n=1}^\infty a_n\,\phi(b_nt)\,e_n.
\]
Suppose that $\lim_n b_n=\infty$. If $F$ is Lipschitz on $\RR$, then it is differentiable almost everywhere on $\RR$. Moreover, $F$ is Lipschitz on $\RR$ if and only if
\begin{equation*}
\sum_{n=1}^\infty (a_nb_n)^p<\infty.
\end{equation*}
\end{theorem}

\begin{proof}
There exist constants $c>0$ and $\delta>0$ such that $|\phi(u)|\ge c|u|$ whenever $|u|<\delta$. Fix $N\in\NN$ and choose $h\neq 0$ such that
\[
|h|<\frac{\delta}{\max_{1\leq n\leq N}b_n}.
\]
Then $b_n|h|<\delta$ for every $1\le n\le N$, and hence
\[
|\phi(b_nh)|\ge c\,b_n|h|,
\qquad 1\le n\le N.
\]
Consequently,
\[
\|F(h)\|_p^p
\ge \sum_{n=1}^N |a_n|^p\,|\phi(b_nh)|^p
\ge c^p |h|^p \sum_{n=1}^N (a_nb_n)^p.
\]
We have $F(0)=0$. Hence, if $F$ is Lipschitz,
\[
\|F(h)\|_p\le \Lip(F) \abs{h}.
\]
Therefore,
\[
\sum_{n=1}^N (a_nb_n)^p \le \frac{\Lip^p(f)}{c^p}.
\]
Since this estimate is valid for every $N\in\NN$, \eqref{thm:decoupled-ae-diff} holds.

Conversely, for $s$, $t\in\RR$ we have
\begin{align*}
\Vert F(t)-F(s)\Vert_{p}^{p}
&=\sum_{n=1}^{\infty}|a_{n}|^{p}|\phi(b_{n}t)-\phi(b_{n} s)|^{p}\\
&\le |\Lip(\phi)|^{p}|t-s|^{p}\sum_{n=1}^{\infty}|a_{n}b_{n}|^{p}.
\end{align*}
So, if \eqref{thm:decoupled-ae-diff} holds, then $F$ is Lipschitz.

Assume from now on that \eqref{thm:decoupled-ae-diff} holds. Since $\phi$ is Lipschitz on $\RR$, it is differentiable almost everywhere on $\RR$. Let $E\subset \RR$ be the set of points $t\in \RR$ such that $\phi'(b_n t)$ exists for every $n\in\NN$. Since each set $\{t\in I\tq \phi'(b_n t)\ \text{exists}\}$ has full measure in $I$, the countable intersection defining $E$ also has full measure in $\RR$.

Define the candidate derivative
\[
f(t):=\sum_{n=1}^\infty a_n b_n \phi'(b_n t)\,e_n.
\]
Since $\phi'$ is bounded, $f(t)\in \ell_p$ a.e.\@ $t\in \RR$.

Fix $t\in E$. For each $n$, define
\[
G_n(h):=a_n \left( \frac{\phi(b_n(t+h))-\phi(b_n t)}{h} -b_n\phi'(b_n t)\right), \qquad h\neq 0,
\]
so that
\[
\frac{F(t+h)-F(t)}{h}-f(t)
=\sum_{n=1}^\infty G_n(h) e_n,
\]
and therefore
\[
\left\|\frac{F(t+h)-F(t)}{h}-f(t)\right\|_p^p
=\sum_{n=1}^\infty \big|G_n(h)\big|^p.
\]
For each fixed $n$ we have $\lim_{h\to 0} G_n(h)=0$.
Since $|\phi'(b_n\, t)|\le \Lip(\phi)$,
\[
|G_n(h)|
\le \Lip(\phi) a_n b_n +|\phi'(b_nt)| a_n b_n \le 2\, \Lip(\phi) a_n b_n
\]
for every $h\neq 0$. The right-hand side is $p$-summable by hypothesis, hence by the dominated convergence for series we may pass the limit inside the sum and obtain
\[
\lim_{h\to 0}\left\|\frac{F(t+h)-F(t)}{h}-f(t)\right\|_p^p = 0.
\]
This proves that $F$ is differentiable at every $t\in E$.
\end{proof}

\begin{remark}\label{r:counterexample-nonflat}
The assumption $\phi'(0)\neq 0$ in the previous corollary cannot be removed. Without that extra hypothesis, for each $0<p<1$ we can construct examples of Lipschitz maps $F\tp\RR\to \ell_p$ of the form
\[
F(t)=\sum_{n=1}^\infty a_n \phi(b_n t)e_n
\]
with $b_{n}$ increasing to $\infty$ such that $\sum_{n=1}^\infty (a_nb_n)^p=\infty$. Indeed, let
\[
\phi(u)=
\begin{cases}
0, & \mbox{if } u\le 1,\\
u-1, & \mbox{if } 1<u<2,\\
1, & \mbox{if } u\ge 2.
\end{cases}
\]
Then $\phi$ is $1$-Lipschitz and differentiable at $0$, with $\phi'(0)=0$. Now set
\[
a_n=2^{-n}, \qquad b_n=2^n, \qquad n\in\NN,
\]
and define
\[
F\tp\RR\to\ell_p,
\qquad
F(t)=\sum_{n=1}^\infty 2^{-n}\phi(2^n t)e_n.
\]
Since $0\le \phi\le 1$ and $(2^{-n})_{n=1}^\infty\in \ell_p$, the map $F$ is well defined. Moreover,
\[
a_n b_n=1 \qquad\text{for every }n,
\]
and therefore
\[
\sum_{n=1}^\infty (a_nb_n)^p=\sum_{n=1}^\infty 1=\infty.
\]
We claim that $F$ is Lipschitz. For each $n\in\NN$, put
\[
F_n(t):=2^{-n}\phi(2^n t), \qquad t\in\RR,
\]
and $J_n=[2^{-n},2^{1-n}]$. We observe the following properties.
\begin{enumerate}[label=(F.\arabic*)]
\item\label{F:a} $F_n(t) \in[0, 2^{-n}]$ for all $t\in\RR$.
\item\label{F:b} $F_n$ is $1$-Lipschitz.
\item\label{F:c} If $s$, $t\in J_n$, then $F_n(t)-F_n(s)=t-s$.
\item\label{F:h} $F_n$ is constant on every interval contained in $\RR\setminus J_n$.
\item\label{F:g} $\abs{J_n}=2^{-n}$ for all $n\in\NN$.
\end{enumerate}
Further, given $k$, $n\in\NN$, since the intervals $(J_n)_{n=1}^\infty$ have pairwise disjoint interiors,
\begin{enumerate}[label=(F.\arabic*),resume]
\item\label{F:d} if $k\not=n$ then $F_n$ is constant on $J_k$; and
\item\label{F:e} If $k\ge n+2$, $t\in J_n$, and $s\in J_k$, then $t-s\ge 2^{-1-n}$.
\end{enumerate}

Let $-\infty< s<t<\infty$. Choose $m\in\ZZ$ such that
\[
2^{-m}<h:=t-s\le 2^{-m+1}.
\]
We split the sum
\[
\|F(t)-F(s)\|_p^p
=
\sum_{n=1}^\infty |F_n(t)-F_n(s)|^p.
\]
according to whether $n\ge m$ or $n\le m-1$. In the former case, by \ref{F:a},
\[
\sum_{n=m}^\infty |F_n(t)-F_n(s)|^p
\le\sum_{n=m}^\infty 2^{-np}
=\frac{2^{-mp}}{1-2^{-p}}
\le\frac{(t-s)^p}{1-2^{-p}}.
\]

In the latter case, $h\le \abs{J_n}$ by \ref{F:g}. Hence, by \ref{F:h}, $F_n(t)-F_n(s)=0$ unless $[s,t]$ meets $J_n$. Morevoer, by \ref{F:e}, the interval $[s,t]$ meets at most two intervals $J_n$ with $n\le m-1$. Using \ref{F:b} we obtain
\[
\sum_{n=1}^{m-1} |F_n(t)-F_n(s)|^p\le 2 \sup_n |F_n(t)-F_n(s)|^p \le 2 h^p.
\]
Combining the two estimates, we obtain that $F$ is Lipschitz with
\[
\Lip(F) \le 2+\frac{1}{1-2^{-p}}.
\]
\end{remark}

However, this construction does not provide a non-differentiable Lipschitz function.
\begin{proposition}\label{p:remark33-diff}
Let $0<p<1$, and let $F\tp\RR\to\ell_p$ be the map from Remark~\ref{r:counterexample-nonflat}, namely $F(t)=\sum_{n=1}^\infty F_n(t)e_n$, where
\[
F_n(t)=
\begin{cases}
0, & t\le 2^{-n},\\
t-2^{-n}, & 2^{-n}<t<2^{1-n},\\
2^{-n}, & 2^{1-n}\le t.
\end{cases}
\]
Then $F$ is differentiable at every point
\[
t\in \RR \setminus \enpar{ \{0,1 \} \cup \{2^{-m} \tq m\in\NN\}}.
\]
More precisely, if
\[
2^{-n}<t<2^{1-n},
\]
for some $n\in\NN$, then
\[
F'(t)=e_n.
\]
On the other hand, $F$ is not differentiable at $0$ nor at any dyadic point
\[
t=2^{-m}, \qquad m\in\NN\cup\{0\}.
\]
In fact, $F$ is not even right differentiable at $0$ either.
\end{proposition}

\begin{proof}
By \ref{F:h}, $F$ is constant on the intervals $(-\infty,0)$ and $(1,\infty)$. Hence, $F'(t)=0$ for all $t\in \RR\setminus[0,1]$, and
\[
\lim_{s\to 0^{-}} \frac{F(s)-F(0)}{s-0}=\lim_{s\to 1^{+}} \frac{F(s)-F(1)}{s-1}=0.
\]
Given $t\in(0,1]$, we use \ref{F:c} and \ref{F:d} to compute $F(s)-F(t)$ for $s$ in neighborhood of $t$. Fix $t\in (0,1)\setminus \{2^{-m} \tq m\in\NN\}$. Then there is a unique $n\in\NN$ such that
\[
2^{-n}<t<2^{1-n}.
\]
If $s\in(2^{-n},2^{1-n})$ then,
\[
F(s)-F(t)=(s-t) e_n.
\]
Thus $F$ is differentiable at $t$ and $F'(t)=e_n$.

Now fix $m\in\NN$ and let $t=2^{-m}$. We have
\[
F(s)-F(t)=\begin{cases} (s-t) e_{m+1},& \mbox{ if } 2^{-m-1}<s<t, \\ (s-t) e_m, & \mbox{ if } t<s<2^{-m+1}.\end{cases}
\]
Since $e_m\neq e_{m+1}$, the difference quotients have different left and right limits. Hence $F$ is not differentiable at $t$.

If $t=1$, and $s\in(1/2,1)$, then $F(s)-F(t)=(s-t) e_1$. Since $e_1\not=0$, $F$ is not differentiable at $1$ either.

Finally, consider the endpoint $0$. Given $n$, $m\in\NN$,
\[
F_n(2^{-m})=\begin{cases} 0, & \mbox{ if } n \le m, \\ 2^{-n}, & \mbox{ if } n > m.\end{cases}
\]
Therefore, since $F(0)=0$,
\[
\frac{F(2^{-m})-F(0)}{2^{-m}-0}
=
f_m:=\sum_{n=1}^\infty 2^{-n} e_{n+m}.
\]
Given $m$, $k\in\NN$ with $m>k$,
\[
\norm{f_m-f_k}_p \ge \abs{e_{k+1}^*(f_m-f_k)} = \frac{1}{2}.
\]
Consequently, $(f_m)_{m=1}^\infty$ is not a Cauchy sequence. Hence, $F$ is not right differentiable at $0$.
\end{proof}
\section{Metric differentiability and rigidity in \texorpdfstring{$\ell_p$}{}}\label{sec:metric}\noindent
In the preceding sections we studied differentiability of Lipschitz mappings $F\tp I\to \ell_p$, $0<p<1$, with respect to the quasi-distance induced by the quasi-norm.
In this section we adopt a different point of view and investigate \emph{metric differentiability} of Lipschitz mappings into $\textsf{F}$-spaces, with special emphasis on the metric space $(\ell_p,\norm{\cdot}_p^p)$.

When $\ell_{p}$ is equipped with the natural metric
\[
(x,y) \mapsto \Vert x-y\Vert_{p}^{p}=\sum_{n=1}^\infty \abs{x_n-y_n}^p, \qquad x=(x_n)_{n=1}^\infty, y=(y_n)_{n=1}^\infty\in \ell_{p},
\]
the resulting metric space is a \emph{snowflake} of a quasi-normed space. We shall see that, in sharp contrast with the quasi-norm setting, the metric geometry of $(\ell_p,\norm{\cdot}_p^p)$ is extremely rigid: every Lipschitz curve from an interval into this space must be constant.

\begin{theorem}\label{thm:metric-rigidity}
Let $0<p<1$, $X$ be a $p$-Banach space with the point separation property, and $F\tp I\to (X,\norm{\cdot}_X^p)$ be a Lipschitz mapping. Then $F$ is constant on $I$.
\end{theorem}

\begin{proof}
Let $L>0$ be such that
\[
\|F(t)-F(s)\|_X^{p}\le L|t-s|,\qquad s,t\in I.
\]
Taking $p$-roots yields
\[
\|F(t)-F(s)\|_X \le L^{1/p}|t-s|^{1/p}.
\]
Consequently, if $J_X\tp X \to \widehat{X}$ is the envelope map,
\[
\norm{J_X(F(t))-J_X(F(s))}_{\widehat{X}}\le L^{1/p}|t-s|^{1/p}\qquad s,t\in I.
\]
Since $1/p>1$, this shows that the Banach-valued function $J_X \circ F$ is H\"older continuous with exponent strictly greater than one. It is well known, and easy to verify directly, that any such mapping from an interval into a normed space must be constant. Indeed, partitioning an interval into arbitrarily many subintervals and using the triangle inequality forces all increments to vanish. Therefore, since $J_X$ is one-to-one, $F$ is constant.
\end{proof}

We emphasize that the assumption that the dual space $X^*$ separates the point of $X$ is essential in Theorem~\ref{thm:metric-rigidity}. Indeed the function $F$ defined as in \eqref{eq:LipLp} is Lipschitz when regarded as a curve into $(L_p[0,1], \norm{\cdot}_{L_p}^p)$.

We obtain as an immediate consequence of Theorem~\ref{thm:metric-rigidity} that metric differentiability in $(\ell_p,\norm{\cdot}_p^p)$ is entirely trivial.

\begin{corollary}
Let $0<p<1$, $X$ be a $p$-Banach space with the point separation property, and $F\tp I\to (X,\norm{\cdot}_X^p)$ be a Lipschitz mapping. Then $\md(F)(t)=0$ for all $t\in I$.
\end{corollary}

Thus, while Theorem~\ref{thm:Kir} guarantees the existence of metric derivatives almost everywhere for Lipschitz curves in arbitrary metric spaces, in the specific case of $(\ell_p,\norm{\cdot}_p^p)$ this derivative always vanishes, because no nonconstant Lipschitz curves exist.

The rigidity exhibited by $(\ell_p,\norm{\cdot}_p^p)$ stands in sharp contrast with the behavior observed in the $p$-norm geometry of $\ell_p$. In Sections~\ref{sec:decoupled} and earlier, we showed that Lipschitz mappings $F\tp I\to\ell_p$ may be highly nontrivial, and that their differentiability properties depend on delicate summability and integrability conditions.

This contrast highlights a fundamental point: in the nonlocally convex setting, differentiability phenomena are extremely sensitive to the choice of geometry. While the metric $(\ell_p,\norm{\cdot}_p^p)$ completely suppresses nontrivial Lipschitz curves, the $p$-norm geometry allows rich behavior, including subtle interactions between coordinatewise differentiation and global convergence of difference quotients. From this perspective, metric differentiability and $p$-norm differentiability address fundamentally different questions and results in one framework should not be expected to carry over to the other.
\section*{Statements and Declarations}
\subsection*{Conflict of Interest}
The authors declare that they have no conflict of interest.
\subsection*{Data Availability}
Since no datasets were generated or analyzed during the current study, data sharing does not apply to this article.
\subsection*{Generative AI use}
AI tools were used during the preparation of the manuscript for literature-related assistance, language polishing, and preliminary exploration.


\begin{bibsection}
\begin{biblist}

\bib{ACP2010}{inproceedings}{
AUTHOR = {Alberti, Giovanni and Cs\"ornyei, Marianna and Preiss, David},
TITLE = {Differentiability of {L}ipschitz functions, structure of null
sets, and other problems},
BOOKTITLE = {Proceedings of the {I}nternational {C}ongress of
{M}athematicians. {V}olume {III}},
PAGES = {1379--1394},
PUBLISHER = {Hindustan Book Agency, New Delhi},
YEAR = {2010},
ISBN = {978-81-85931-08-3; 978-981-4324-33-5; 981-4324-33-7},
}

\bib{AlbiacJOCA}{article}{
author={Albiac, F.},
title={The role of local convexity in Lipschitz maps},
journal={J. Convex Anal.},
volume={18},
date={2011},
number={4},
pages={983--997},
}

\bib{AlbiacAnsorena2012}{article}{
author={Albiac, F.},
author={Ansorena, J. L.},
title={On a problem posed by M. M. Popov},
journal={Studia Math.},
volume={211},
date={2012},
number={3},
pages={247--258},
}

\bib{AlbiacAnsorenaJFA1}{article}{
author={Albiac, F.},
author={Ansorena, J. L.},
title={Integration in quasi-Banach spaces and the fundamental theorem of calculus},
journal={J. Funct. Anal.},
volume={264},
date={2013},
number={9},
pages={2059--2076},
}

\bib{AlbiacAnsorenaJFA2}{article}{
author={Albiac, F.},
author={Ansorena, J. L.},
title={Optimal average approximations for functions mapping in quasi-Banach spaces},
journal={J. Funct. Anal.},
volume={266},
date={2014},
number={6},
pages={3894--3905},
}

\bib{AlbiacAnso2016}{article}{
author={Albiac, F.},
author={Ansorena, J. L.},
title={On Lipschitz maps, martingales, and the Radon-Nikod\'ym property
for ${F}$-spaces},
journal={Mediterr. J. Math.},
volume={13},
date={2016},
number={4},
pages={1963--1980},
issn={1660-5446},
}

\bib{AABW2021}{article}{
author={Albiac, Fernando},
author={Ansorena, Jos\'{e}~L.},
author={Bern\'{a}, Pablo~M.},
author={Wojtaszczyk, Przemys{\l}aw},
title={Greedy approximation for biorthogonal systems in quasi-{B}anach
spaces},
date={2021},
journal={Dissertationes Math. (Rozprawy Mat.)},
volume={560},
pages={1\ndash 88},
}

\bib{AlbiacKalton2016}{book}{
author={Albiac, Fernando},
author={Kalton, Nigel~J.},
title={Topics in {B}anach space theory},
edition={Second Edition},
series={Graduate Texts in Mathematics},
publisher={Springer, [Cham]},
date={2016},
volume={233},
ISBN={978-3-319-31555-3; 978-3-319-31557-7},
url={https://doi.org/10.1007/978-3-319-31557-7},
note={With a foreword by Gilles Godefroy},
review={\MR{3526021}},
}

\bib{Assouad1983}{article}{
AUTHOR = {Assouad, Patrice},
TITLE = {Plongements lipschitziens dans {${\bf R}\sp{n}$}},
JOURNAL = {Bull. Soc. Math. France},
VOLUME = {111},
YEAR = {1983},
NUMBER = {4},
PAGES = {429--448},
ISSN = {0037-9484},
URL = {http://www.numdam.org/item?id=BSMF_1983__111__429_0},
}

\bib{Bate_Structure_of_measures_2015}{article}{
AUTHOR = {Bate, David},
TITLE = {Structure of measures in {L}ipschitz differentiability spaces},
JOURNAL = {J. Amer. Math. Soc.},
VOLUME = {28},
YEAR = {2015},
NUMBER = {2},
PAGES = {421--482},
ISSN = {0894-0347,1088-6834},
DOI = {10.1090/S0894-0347-2014-00810-9},
URL = {https://doi.org/10.1090/S0894-0347-2014-00810-9},
}

\bib{BenLin}{book}{
author={Benyamini, Y.},
author={Lindenstrauss, J.},
title={Geometric nonlinear functional analysis. Vol. 1},
series={American Mathematical Society Colloquium Publications},
volume={48},
publisher={American Mathematical Society},
place={Providence, RI},
date={2000},
pages={xii+488},
}

\bib{Birkhoff1935}{article}{
author={Birkhoff, Garrett},
issn={0002-9947},
issn={1088-6850},
doi={10.2307/1989687},
review={Zbl 0013.00803},
title={Integration of functions with values in a Banach space},
journal={Transactions of the American Mathematical Society},
volume={38},
pages={357--378},
date={1935},
publisher={American Mathematical Society (AMS), Providence, RI},
}

\bib{Bochner1933}{article}{
author={Bochner, Salomon},
language={German},
title={Integration von Funktionen, deren Werte die Elemente eines Vektorraumes sind},
journal={Fundamenta Mathematicae},
volume={20},
pages={262--276},
date={1933},
}

\bib{Bochner1933b}{article}{
author={Bochner, S.},
language={German},
title={Absolut-additive abstrakte Mengenfunktionen.},
journal={Fundamenta Mathematicae},
volume={21},
pages={211--213},
date={1933},
publisher={Polish Academy of Sciences (Polska Akademia Nauk - PAN), Institute of Mathematics (Instytut Matematyczny), Warsaw},
}

\bib{Clar}{article}{
author={Clarkson, J. A.},
title={Uniformly convex spaces},
journal={Trans. Amer. Math. Soc.},
volume={40},
date={1936},
number={3},
pages={396--414},
}

\bib{HeinrichMankiewicz1982}{article}{
author={Heinrich, S.},
author={Mankiewicz, P.},
title={Applications of ultrapowers to the uniform and Lipschitz classification of Banach spaces},
journal={Studia Math.},
volume={73},
date={1982},
number={3},
pages={225--251},
}

\bib{DunMor}{article}{
author={Dunford, N.},
author={Morse, A. P.},
title={Remarks on the preceding paper of James A. Clarkson: ``Uniformly
convex spaces'' [Trans.\ Amer.\ Math.\ Soc.\ {\bf 40} (1936), no.\ 3; 1
501 880]},
journal={Trans. Amer. Math. Soc.},
volume={40},
date={1936},
number={3},
pages={415--420},
}

\bib{Gromov_metric_structures_1999}{book}{
AUTHOR = {Gromov, Misha},
TITLE = {Metric structures for {R}iemannian and non-{R}iemannian
spaces},
SERIES = {Progress in Mathematics},
VOLUME = {152},
NOTE = {Based on the 1981 French original [MR0682063 (85e:53051)],
With appendices by M.\ Katz, P.\ Pansu and S.\ Semmes,
Translated from the French by Sean Michael Bates},
PUBLISHER = {Birkh\"auser Boston, Inc., Boston, MA},
YEAR = {1999},
PAGES = {xx+585},
ISBN = {0-8176-3898-9},
}

\bib{Hunt1966}{article}{
author={Hunt, R.~A.},
title={On {$L(p,\,q)$} spaces},
date={1966},
ISSN={0013-8584},
journal={Enseign. Math. (2)},
volume={12},
pages={249\ndash 276},
review={\MR{223874}},
}

\bib{Hyers1939}{article}{
author={Hyers, D.~H.},
title={Locally bounded linear topological spaces},
date={1939},
ISSN={0034-7760},
journal={Rev. Ci. (Lima)},
volume={41},
pages={555\ndash 574},
}

\bib{Kalton1980b}{article}{
author={Kalton, Nigel~J.},
title={Linear operators on {$L_{p}$} for {$0<p<1$}},
date={1980},
ISSN={0002-9947},
journal={Trans. Amer. Math. Soc.},
volume={259},
number={2},
pages={319\ndash 355},
url={https://doi-org/10.2307/1998234},
review={\MR{567084}},
}

\bib{KaltonZero}{article}{
author={Kalton, Nigel~J.},
title={Curves with zero derivative in $\mathsf{F}$-spaces},
journal={Glasgow Math. J.},
volume={22},
date={1981},
number={1},
pages={19--29},
issn={0017-0895},
}

\bib{KaltonRNP}{article}{
author={Kalton, Nigel~J.},
title={An analogue of the Radon-Nikod\'ym property for nonlocally convex quasi-Banach spaces},
journal={Proc. Edinburgh Math. Soc. (2)},
volume={22},
date={1979},
number={1},
pages={49--60},
}

\bib{Kalton1986}{article}{
author={Kalton, Nigel~J.},
title={Banach envelopes of nonlocally convex spaces},
date={1986},
ISSN={0008-414X},
journal={Canad. J. Math.},
volume={38},
number={1},
pages={65\ndash 86},
url={https://doi.org/10.4153/CJM-1986-004-2},
review={\MR{835036}},
}

\bib{Kalton2008}{misc}{
author={Kalton, Nigel~J.},
title={Personal communication to the first-named author},
date={2008},
}

\bib{KPR1985}{book}{
author={Kalton, Nigel~J.},
author={Peck, N. T.},
author={Roberts, James W.},
title={An $F$-space sampler},
series={London Mathematical Society Lecture Note Series},
volume={89},
publisher={Cambridge University Press, Cambridge},
date={1984},
pages={xii+240},
}

\bib{Kirchheim1994}{article}{
author={Kirchheim, Bernd},
title={Rectifiable metric spaces: local structure and regularity of the Hausdorff measure},
journal={Proc. Amer. Math. Soc.},
volume={121},
date={1994},
number={1},
pages={113--123},
}

\bib{Rudin1991}{book}{
author={Rudin, Walter},
title={Functional analysis},
edition={Second},
series={International Series in Pure and Applied Mathematics},
publisher={McGraw-Hill, Inc., New York},
date={1991},
}

\bib{Turpin1976}{article}{
author={Turpin, Philippe},
title={Convexit\'{e}s dans les espaces vectoriels topologiques
g\'{e}n\'{e}raux},
date={1976},
ISSN={0012-3862},
journal={Dissertationes Math. (Rozprawy Mat.)},
volume={131},
pages={221},
}

\bib{Vogt1967}{article}{
author={Vogt, D.},
title={Integrationstheorie in $p$-normierten R\"aumen},
language={German},
journal={Math. Ann.},
volume={173},
date={1967},
pages={219--232},
}
\end{biblist}
\end{bibsection}
\end{document}